\documentclass[11pt,a4paper]{article} 
\usepackage[latin1]{inputenc}
\usepackage{indentfirst}
\usepackage{amsmath}
\usepackage{amsfonts}
\usepackage{amsthm}
\usepackage{amssymb}
\usepackage{graphics}
\usepackage{graphicx}
\usepackage{multicol}
\usepackage{color}
\usepackage{authblk}
\usepackage{epsfig}

\def\N{\mathbb{N}} 

\def\P{\mathbb{P}}

\def\E{\mathbb{E}}

\def\R{\mathbb{R}}

\def\w{\widetilde}

\def\ud{\mathrm{d}}

\newcommand{\be} {\begin{equation}}
\newcommand{\ee} {\end{equation}}
\newcommand{\bea} {\begin{eqnarray}}
\newcommand{\eea} {\end{eqnarray}}
\newcommand{\Bea} {\begin{eqnarray*}}
\newcommand{\Eea} {\end{eqnarray*}}

\newtheorem{Thm}{Theorem}
\newtheorem{Lem}[Thm]{Lemma}

\newtheorem{Prop}[Thm]{Proposition}
\newtheorem{Cor}[Thm]{Corollary}

\theoremstyle{definition} 
\theoremstyle{definition} 
\theoremstyle{definition} \newtheorem*{key}{Key words}
\theoremstyle{definition} \newtheorem*{ams}{A.M.S. Classification}

\theoremstyle{remark}\newtheorem{Rque}{Remark}

\begin{document}
\title{On the extinction of Continuous State Branching Processes with catastrophes}
\author[1]{Vincent Bansaye \thanks{vincent.bansaye@polytechnique.edu}}
\author[2]{Juan Carlos Pardo Millan \thanks{jcpardo@cimat.mx}}
\author[1,3]{Charline Smadi \thanks{charline.smadi@polytechnique.edu}}
\affil[1]{CMAP, \'Ecole Polytechnique, Route de Saclay, F-91128 Palaiseau Cedex, France}
\affil[2]{CIMAT A.C. Calle Jalisco s/n. C.P. 36240, Guanajuato, Mexico}
\affil[3]{Universit\'e Paris-Est, CERMICS (ENPC), 6-8 Avenue Blaise Pascal, Cit\'e
Descartes, F-77455 Marne-la-Vallée, France}
%   CERMICS, Universit´ Paris-Est, 6-8 av. Blaise Pascal, Champs-sur-eMarne, 77455 Marne La Vall´e, FRANCE.                          e

%
\maketitle \vspace{1.5cm}

\begin{abstract}
We consider continuous state branching processes (CSBP)  with additional multiplicative jumps   modeling dramatic  events in a random environment. 
These jumps  are described by a L\'evy process with bounded variation  paths. 
We construct  a process of this class  as the unique solution of a stochastic differential equation.  
 The quenched branching property of the process allows us to derive quenched and annealed results  and to  observe new asymptotic behaviors. 
We characterize the  Laplace exponent of the process as the solution of a backward ordinary differential equation and establish  the probability of extinction.
 Restricting our attention to the critical and subcritical cases, we show that   four regimes  arise for the speed of extinction, as in the case  of branching processes in random environment in discrete time and space.
The proofs   are based on the precise asymptotic behavior of exponential functionals of L\'evy processes. Finally, we apply these results to  a cell infection model and determine the mean speed of propagation of the infection.
\end{abstract}

\begin{key} Continuous State Branching Processes, L\'evy processes, Poisson Point Processes, %{\color{red} Catastrophe modeling}, 
Random Environment, Extinction, Long time behavior
\end{key}
\begin{ams} 60J80, 60J25, 60G51, 60H10, 60G55, 60K37.
\end{ams}

\section{Introduction}
  Continuous state branching processes (CSBP) are the analogues of Galton-Watson (GW) processes in continuous time and continuous state space. They have  been introduced by Jirina \cite{MR0101554} and studied by many authors  including Bingham \cite{MR0410961}, Grey \cite{MR0408016}, Grimvall \cite{MR0362529}, Lamperti \cite{MR0208685,MR0217893}, to name but a few.

	A CSBP $Z=(Z_t, t\geq 0)$ is a strong Markov process taking values in $[0,\infty]$, where $0$ and $\infty$ are absorbing states,  and satisfying the branching property.  We denote by  $(\mathbb{P}_x, \, x> 0)$  the law of $Z$ starting from $x$. Lamperti  \cite{MR0217893}  proved  that there is a bijection between CSBP and scaling limits of GW processes.
Thus they may  model the evolution of renormalized large populations on a large time scale.

The branching property implies that the Laplace transform of $Z_t$ is of the form
\begin{equation*}
\mathbb{E}_x\Big[ \exp(-\lambda  Z_t)\Big]=\exp\{-x
u_t(\lambda)\},\qquad\textrm{for }\lambda\geq 0,
\label{CBut}
\end{equation*}
for some non-negative function $u_t$. According to Silverstein \cite{MR0226734}, this function is determined by the
integral equation
\begin{equation*}
\int_{u_t(\lambda)}^\lambda \frac{1}{\psi(u)}{\rm d} u=t,
\label{DEut}
\end{equation*}
where $\psi$ is known as the  branching mechanism associated to $Z$. We assume here that $Z$ has finite mean, so that  we have the following classical representation  
%the  L\'evy-Khintchine formula ensures that   
\begin{equation}\label{defpsi}\psi(\lambda) = -g \lambda + \sigma^2 \lambda^2 + \int_0^{\infty} \left( e^{-\lambda {z}}-1+\lambda {z} 
\right) \mu(\ud {z}),\end{equation}
where $g\in \R$, $\sigma\geq 0$ and $\mu$ is a $\sigma$-finite measure on $(0,\infty)$ such that $\int_{(0,\infty)}\big({z}\land {z}^2\big)\mu(\ud {z})$ is finite.
%The function $\psi$ is the branching mechanism of $Z$. 
The CSBP is then characterized by the triplet $(g, \sigma, \mu)$ and can also be defined as the unique  non-negative strong solution of a  stochastic differential equation. More precisely, from Fu and Li \cite{RePEc:eee:spapps:v:120:y:2010:i:3:p:306-330} we have
\begin{equation}\label{defCSBP}
Z_t=Z_0+\int_0^t gZ_s \ud s +\int_0^t\sqrt{2\sigma^2 Z_s} \ud B_s+\int_0^t\int_0^\infty\int_0^{Z_{s-}} z\widetilde{N}_0(\ud s, \ud z, \ud u),
\end{equation}
where $B$ is a standard Brownian motion, $N_0(\ud s, \ud z, \ud u)$ is a Poisson random measure with intensity $\ud s\mu(\ud z)\ud u$ independent of $B$,  and  $\widetilde{N}_0$ is the  compensated measure of $N_0$. 

The stable case with drift, i.e. $\psi(\lambda)=-g\lambda+c\lambda^{1+\beta}$, with $\beta $ in $(0,1]$,  corresponds to the CSBP  that one can obtain by scaling limits of GW processes with a fixed   reproduction law.   It is  of special interest in this paper since the Laplace exponent can be computed explicitly and  it can also  be used to derive asymptotic results for more general  cases. \\

 In this work, we are interested in modeling catastrophes which occur at random and kill each individual with some probability (depending on the catastrophe). In terms of the CSBP representing the scaling limit of the size of a large population, this amounts to  letting the process make a negative jump, i.e. multiplying its current value by a random fraction.
The process that we obtain is still Markovian whenever the catastrophes  follow a time homogeneous Poisson Point Process. Moreover, we show that conditionally on the times and the effects of the catastrophes, the process %{\color{red}of interest} 
satisfies the branching property. Thus,  it yields a particular class of CSBP in  random
 environment, which can also be obtained as the scaling limit of GW processes in random environment (see \cite{bansim}). Such processes are motivated in particular by a cell division
 model; see for instance \cite{MR2754402} and Section \ref{appli}. 
 
 We  also consider positive jumps that may represent   immigration events   proportional to the size of the current  population. Our motivation comes from  the aggregation behavior of some species. We refer to Chapter 12 in \cite{dan} for adaptive explanations of these aggregation behaviors, or \cite{rode2013join} which shows that aggregation behaviors may result from manipulation by parasites to increase their transmission. For convenience, we still call these dramatic events  catastrophes.

  The process $Y$ that we consider in this paper  is then  called  \emph{a CSBP with catastrophes}. Roughly speaking, it can be defined
  as follows: The process
$Y$ follows the SDE   (\ref{defCSBP}) between  catastrophes, which are then given in terms of  the jumps of a L\'evy process with bounded variation paths. Thus the set of times at which catastrophes occur may have accumulation points, but the mean effect of the catastrophes   has a finite first moment. When a catastrophe with effect  $m_t$ occurs at time $t$, we have
$$Y_{t}%-Y_{t-}
=m_t Y_{t-}.$$  We defer the formal definitions to Section \ref{csbpc}. We also note that Brockwell has considered birth and death branching processes with  another kind of catastrophes, see e.g. \cite{Br}. 
%$$dY_t=dZ_t+\sum_

 First we verify that  CSBP with catastrophes  are well defined as  solutions of a certain stochastic differential equation,  which we give as (\ref{EDS}). We   characterize their Laplace exponents via an  ordinary differential equation (see Theorem \ref{th1}), which allows us to describe  their long time behavior. In particular,
we prove an extinction criterion for the CSBP with catastrophes  which is given in terms of the sign of $\E[g+\sum_{s\leq 1}\log m_s ]$. We also establish a central limit theorem conditionally on  survival and  under some moment assumptions (Corollary \ref{TCL}). 

 We   then  focus on the case when the branching mechanism associated  to the CSBP with catastrophes $Y$ has the form $\psi(\lambda)=-g\lambda+c\lambda^{1+\beta}$, for $\beta\in (0, 1]$, i.e.  the stable case.  In this scenario,  the extinction and absorption events coincide, which  means that $\{\lim_{t \to \infty}Y_t=0\}=\{\exists t \geq 0, Y_t=0\}$. We prove that the speed of extinction is directly related to the asymptotic behavior  of   exponential  functionals of   L\'evy processes  (see Proposition \ref{prop}). More precisely, we show that the extinction probability of a stable CSBP with catastrophes can be expressed as follows:
\[ \P(Y_t>0)=\mathbb{E}\left[F\left(\int_0^t e^{-\beta K_s}\ud s\right)\right],
\]
where $F$ is a function with a particular asymptotic behavior and $K_t:=gt+\sum_{s\leq t}\log m_s $ is a L\'evy process of bounded variation that does 
not drift to $+\infty$ and satisfies  an exponential positive moment condition. We  establish the asymptotic behavior  of  the survival probability   
(see Theorem \ref{funcLev}) and  find four different regimes  when this probability is equal  to zero. Actually, such asymptotic behaviors have 
previously been found for branching processes in random environments in discrete time and space (see e.g. \cite{MR1821473, MR1983172, Afanasyev2005}). 
Here, the regimes depend  on the shape of the Laplace exponent of $K$, i.e. {on} the drift $g$ of the CSBP and the law of the catastrophes. The asymptotic 
behavior of exponential functionals of L\'evy processes drifting to $+\infty$ has been deeply studied by many authors, see for instance Bertoin and Yor 
\cite{MR2178044} and  references therein. To our knowledge, 
the remaining  cases have been studied only by  Carmona et al.  (see  Lemma 4.7 in \cite{MR1648657})  but their result focuses only on one regime. Our 
result is closely related to the discrete framework via the asymptotic behaviors of functionals
 of random walks. More precisely, we use in our arguments local limit theorems for semi direct products
 \cite{MR1443957, MR1821473} and some analytical results on random walks \cite{koz,MR1633937}, see Section \ref{proof}.   \\

 From the speed of extinction in the stable case, we can deduce the speed of extinction of a larger class of CSBP with catastrophes satisfying  the condition  that extinction and absorption coincide (see Corollary \ref{cor}). General results for  the case of L\'evy processes of unbounded variation do not seem easy to obtain since the existence of the process $Y$ and our {approximation} methods are  not so  easy to deduce. The particular case when 
  $\mu=0$ and the environment  $K$  is given  by a  Brownian motion  has been studied in \cite{BH}. The authors in \cite{BH}  also obtained similar asymptotics regimes using the explicit law of  $\int_0^t \exp(-\beta K_s)ds$. 

Finally, we apply  our results to a cell infection model introduced in  \cite{MR2754402} (see Section \ref{appli}). In this model, the infection in a cell line is given by  a Feller diffusion with catastrophes. We derive here the different possible speeds  of the infection {propagation}. More generally, these results can be related to  some ecological problems concerning  the role of environmental and demographical stochasticit{ies}. Such topics are fundamental in conservation biology, as discussed for instance in  Chapter 1 in \cite{LANDE}. Indeed, the survival of the population may be either due to the randomness of the individual reproduction, which {is} specified in our model  by the parameters $\sigma$ and $\mu$ of the CSBP,  or to the randomness (rate, size) of the catastrophes due to the environment. For a study of relative effects of environmental and demographical stochasticit{ies}, the reader is referred  to \cite{lande1993risks} and references therein.\\

The remainder of the paper is structured as follows. In Section 2, we define and  study the CSBP with catastrophes. Section \ref{speed} is devoted to  the study of the  extinction probabilities where special attention is given to the stable case. In Section \ref{proof}, we analyse the asymptotic behavior of  exponential  functionals of  L\'evy processes of {bounded variation}. This result is  the key to deducing the different extinction regimes. In Section \ref{appli}, we apply our results to a cell infection model.  Finally, Section \ref{annexe}  contains some technical results used in the proofs and deferred for the convenience of the reader.

\section{CSBP with catastrophes}
\label{csbpc}

We consider a CSBP $Z=(Z_t, t\ge 0)$ defined by (\ref{defCSBP}) and characterized by the triplet $(g, \sigma, \mu)$, where  we recall that $\mu$ satisfies
\begin{equation}
 \label{condmu} \int_0^\infty ({z} \land {z}^2) \mu(\ud {z})<\infty. \end{equation}
The catastrophes are independent of the process $Z$ and are given by a Poisson random  measure {$N_1=\sum_{i \in I}\delta_{t_i,m_{t_i}}$} on $[0,\infty)\times[0,\infty)$ with intensity  $\ud t\nu(\ud {m})$ such that
\begin{equation} \label{condh1}
\nu(\{0\})=0\qquad\textrm{and}\qquad 0 <\int_{(0,{\infty})} (1\land \big|{m}-1\big| ) \nu(\ud {m})<\infty.
\end{equation}
\noindent The jump process
$$\Delta_t=\int_0^t \int_{(0,\infty)}\log({m}) N_1(\ud s,\ud {m})= \sum_{s\leq t} \log(m_s) ,$$
 is thus a L\'evy process with paths of bounded variation, which is non identically zero.

The CSBP $(g,\sigma,\mu)$ with catastrophes $\nu$ is defined as the solution of the following stochastic differential equation:
\begin{eqnarray}\label{EDS}
Y_t=Y_0+\int_0^t gY_s \ud s +\int_0^t\sqrt{2\sigma^2 Y_s} \ud B_s
&+&\int_0^t\int_{[0,\infty)}\int_0^{Y_{s-}} z\widetilde{N}_0(\ud s, \ud z, \ud u)\nonumber \\ 
&+&\int_0^t\int_{[0,\infty)} \Big({m}-1\Big) Y_{s-}N_1(\ud s,\ud m), 
\end{eqnarray}
where $Y_0>0$ a.s.

Let $\mathcal{B}\mathcal{V}(\R_+)$ be the set of c\`adl\`ag functions on $\R_+:=[0,\infty)$ of bounded variation and   $C^2_b$ the set of all functions that are twice differentiable  and are bounded together with their derivatives, then the  following result of existence and unicity holds:

\begin{Thm} \label{th1} 
The  stochastic differential equation (\ref{EDS}) has a unique non-negative strong solution  $Y$  for any  $g\in \R,\sigma\ge 0$, $\mu$ and $\nu$ satisfying  conditions (\ref{condmu}) and (\ref{condh1}), respectively. Then, the process $Y=(Y_t, t\ge 0)$ is a c\`adl\`ag Markov process satisfying the branching property conditionally on $\Delta=(\Delta_t, t\ge 0)$ and  its  infinitesimal generator $\mathcal{A}$ satisfies for every $f\in C^2_b$
\begin{equation}\label{defgateur}
\begin{split}
\mathcal{A}f(x)&=gxf^\prime(x)+\sigma^2xf^{\prime\prime}(x)+\int_0^\infty \Big(f({m}x)-f(x)\Big)\nu(d{m})\\
&\hspace{4cm} +\int_0^\infty \Big(f(x+z)-f(x)-zf^\prime(x)\Big)x\mu(dz).
\end{split}
\end{equation}
Moreover, for every $t\geq 0$,
$$\E_y\left[\exp\Big\{-\lambda \exp\big\{-gt-\Delta_t\big\}Y_t\Big\}  \bigg| \ \Delta\right]=\exp\Big\{-yv_t(0,\lambda, \Delta)\Big\} \qquad \text{a.s.},$$
where for every $(\lambda,{\delta}) \in (\R_+,\mathcal{B}\mathcal{V}(\R_+))$, $v_t : s \in [0,t] \mapsto v_t(s, \lambda ,{\delta})$ is the unique solution of the following backward differential equation :
\be
\label{eqnv}
\frac{\partial}{\partial s} v_t(s,\lambda, {\delta})=e^{gs+ {\delta}_s}\psi_0\big(e^{-gs- {\delta}_s}v_{{t}}(s,\lambda,{\delta})\big), \qquad v_t(t,\lambda,{\delta})=\lambda, 
 \ee
and
\begin{equation} \label{defpsi0} \psi_0(\lambda)=\psi(\lambda)-\lambda \psi'(0)=\sigma^2\lambda^2+\int_0^{\infty} (e^{-\lambda {z}}-1+\lambda {z})\mu(d{z}). \end{equation}
\end{Thm}

\begin{proof}

Under Lipschitz conditions, the existence and uniqueness of strong solutions for  {stochastic differential equations} are classical results (see \cite{SDEsApp1}). In our case, the result follows from Proposition 2.2 and Theorems 3.2 and 5.1 in \cite{RePEc:eee:spapps:v:120:y:2010:i:3:p:306-330}. By It\^o's formula (see for instance \cite{SDEsApp1} Th.5.1),   the solution of the SDE (\ref{EDS}), $(Y_t, t\ge 0)$ solves the following  martingale problem.  For every {$f\in C^2_b$},
\[
\begin{split}
f(Y_t)&=f(Y_0)+\textrm{ loc. mart. }+g\int_0^tf^\prime(Y_s)Y_s \ud s\\
&+\sigma^2\int_0^tf^{\prime\prime}(Y_s) Y_s \ud s+\int_0^t\int_0^\infty Y_{s}\Big(f(Y_{s}+z)-f(Y_{s})-f^\prime(Y_s)z\Big)\mu(\ud z)\ud s\\
& +\int_0^t\int_0^\infty\Big(f({m}Y_{s})-f(Y_{s})\Big)\nu(\ud {m})\ud s ,\\
\end{split}
\]
where the local martingale is given by
\begin{eqnarray} \label{loc_mart}
&&\int_0^tf^{\prime}(Y_s) \sqrt{2\sigma^2Y_s} \ud B_s+\int_0^t\int_0^\infty\Big(f({m}Y_{s^-})-f(Y_{s^-})\Big)\widetilde{N}_1(\ud s, \ud {m}) \\ \nonumber
&&\hspace{1cm}+\int_0^t\int_0^\infty\int_{0}^{Y_{s-}}\Big(f(Y_{s-}+z)-f(Y_{s-})\Big)\widetilde{N}_0(\ud s, \ud z, \ud u), 
\end{eqnarray}
and $\widetilde{N}_1$ is the compensated measure of $N_1$.  Even though the process in (\ref{loc_mart}) is a local martingale,  we can define a localized 
version of the corresponding martingale problem as in Chapter 4.6 of Ethier and Kurtz \cite{EK}. We leave the details to the reader.
From pathwise uniqueness, we deduce that  the solution of $(\ref{EDS})$ is a strong Markov process whose generator is given by (\ref{defgateur}).

The branching property of $Y$, conditionally on $\Delta$,  is inherited from the branching property of the CSBP and the fact that the additional jumps are multiplicative. 

To prove the second part of the theorem, let us now work conditionally on $\Delta$. Applying It\^o's formula to the process $\widetilde{Z}_t=Y_t\exp\{-gt-\Delta_t\}$,  we obtain
 $$
 \widetilde{Z}_t=Y_0+\int_0^t e^{-gs-\Delta_s} \sqrt{2\sigma^2Y_s}\ud B_s+\int_0^t\int_0^\infty\int_0^{Y_{s-}} e^{-gs-\Delta_{s-}}z\widetilde{N}_0(\ud s, \ud z, \ud u),
$$
and then $\widetilde{Z}$ is a local martingale conditionally on $\Delta$. A new application of It\^o's formula ensures that for every $F\in C_b^{1,2}$, $F(t,\widetilde{Z}_t)$ is also  a local martingale if and only if for every $t \geq 0$,
\begin{eqnarray}\label{martlocF}
&\displaystyle\int_0^t & \frac{\partial^2}{\partial x^2}F(s,\widetilde{Z}_s)\sigma^2 \tilde{Z}_s e^{-gs-\Delta_s} \ud s+\int_0^t\frac{\partial}{\partial s}F(s,\widetilde{Z}_s) \ud s \\ 
&& \hspace{-0.5cm}+\int_0^t\int_0^\infty \tilde{Z}_{s}\Big(\Big[F(s,\widetilde{Z}_{s}+ze^{-gs-\Delta_{s}})-F(s,\widetilde{Z}_{s})\Big] e^{gs+\Delta_{s}} -\frac{\partial}{\partial x}F(s,\widetilde{Z}_s)z\Big)\mu(\ud z) \ud s =0.\nonumber
\end{eqnarray}
%\marginpar{ {vérifier qu'on n'a pas ici un problème de filtration et qu'on a aussi une martingale locale en travaillant conditionnellement. Si c'est ok, mettre quand meme un mot} }
In the vein of \cite{SDEsApp1, MR2754402}, we  choose  $F(s,x):=\exp\{-xv_t(s,\lambda, \Delta)\}$, where  $v_t(s,\lambda,\Delta)$ is differentiable with respect to the variable $s$, non-negative and such that $v_t(t,\lambda,\Delta)=\lambda$, for $\lambda\ge 0$. The function $F$ is bounded, so that  $(F(s,\tilde{Z}_s), 0\le s\le t)$ will be a martingale if and only if  for every $s\in [0,   t]$
  \[
\frac{\partial}{\partial s}v_t(s,\lambda,\Delta)=e^{gs+\Delta_s} \psi_0\left(e^{-gs-\Delta_s}v_t(s,\lambda,\Delta)\right), \quad \text{a.s.},
 \]
  where $\psi_0$ is defined in (\ref{defpsi0}). 

Proposition \ref{existun} in  Section 6 ensures that a.s. the solution of this backward differential equation exists and is unique, which essentially 
comes from the  Lipschitz property of $\psi_0$ (Lemma \ref{lemv}) and 
the fact that $\Delta$ possesses bounded variation paths.
Then the process $(\exp\{-\tilde{Z}_s v_t(s,\lambda, \Delta)\}, 0\le s\le t)$ is a martingale conditionally on $\Delta$ and 
$$\E_y\left[\exp\Big\{-\tilde{Z}_t v_t(t,\lambda, \Delta)\Big\}  \bigg| \ \Delta\right]=\E_y\left[\exp\Big\{-\tilde{Z}_0 v_t(0,\lambda, \Delta)\Big\}  \bigg| \Delta\right] \quad \text{a.s.},$$
which yields
\begin{equation} \label{lap_exp} \E_y\left[\exp\Big\{-\lambda\tilde{Z}_t \Big\}  \bigg| \ \Delta\right]=\exp\Big\{-y v_t(0,\lambda, \Delta)\Big\} \quad \text{a.s.} \end{equation}
This implies our   result.
\end{proof}

Referring to Theorem 7.2 in \cite{MR2250061}, we recall that a L\'evy process has three possible asymptotic behaviors:  either it drifts to $\infty$,  $-\infty$, or oscillates a.s. In particular, if the L\'evy process  has a finite first moment, the sign of its expectation yields the regimes of above. We extend this classification to CSBP with catastrophes.

\begin{Cor} \label{cor1} We have the {following} three regimes. \\

i) If $(\Delta_t+gt)_{t \geq 0}$ drifts to $-\infty$, then $\P(Y_t\rightarrow 0 \ \vert \ \Delta)=1 $ a.s. \\
 
ii) If $(\Delta_t+gt)_{t \geq 0}$ oscillates, then $\P(\liminf_{t\rightarrow \infty} Y_t=0\ \vert \ \Delta)=1$ a.s. \\

iii) If $(\Delta_t+gt)_{t \geq 0}$ drifts to $+\infty$ and there exists $\varepsilon >0$, such that
\begin{equation} \label{xlogx} \int_{0}^{\infty} {z}\log^{{1+\epsilon}}(1+{z}) \mu(d{z})<\infty,
\end{equation}
then  $\P( \liminf_{t\to\infty} Y_t >0 \ \vert \ \Delta)>0$ a.s. and there exists a non-negative finite r.v. $W$ such that
 $$e^{-gt-\Delta_t}Y_t\xrightarrow[t\rightarrow \infty]{}W \quad a.s., \qquad \{W=0\}=\Big\{\lim_{t\to\infty}Y_t =0\Big\}.$$
\end{Cor}

\begin{Rque}
 In the regime $(ii)$, $Y$ may be absorbed in finite time a.s. (see the next section). But  $Y_t$ may also a.s. {do} not tend  to zero. For example, if $\mu=0$ and $\sigma=0$, then $Y_t=\exp(gt +\Delta_t)$ and
$\limsup_{t\rightarrow \infty} Y_t=\infty$.

Assumption  $(iii)$ of the corollary does not imply that
$\{\lim_{t\to\infty}Y_t =0\}=\{\exists t : Y_t=0\}.$
{Indeed}, the case  $\mu(dx)=x^{-2}\mathbf{1}_{ [0,1]}{(x)}dx$ inspired by Neveu's CSBP yields  $\psi(u)\sim u\log u$ as $u\to \infty$.  Then, according to Remark 2.2 in \cite{MR2466449}, $\P(\exists t : Y_t=0)=0$ and $0<\P(\lim_{t\to\infty}Y_t = 0)<1$.
\end{Rque}

\begin{proof}
We use (\ref{martlocF}) with $F(s,x)=x$ to get that $\tilde{Z}=(Y_t\exp(-gt-\Delta_t) : t\geq 0)$ is a non-negative local martingale. Thus it is a non-negative supermartingale and it
 converges a.s. to a non-negative finite random variable $W$. This implies the proofs of (i-ii).
 
In the case when {$(gt+\Delta_t,t \geq 0)$} goes to $+\infty$, we  prove  that $\P(W>0 \ \vert \ \Delta)>0$ a.s. 
According to Lemma \ref{exg} in Section 6,  the assumptions of (iii) ensure the existence of a non-negative increasing function ${k}$ on $\R^+$ such that for all $\lambda>0,$
$$ \psi_0(\lambda)\leq \lambda {k}(\lambda) \quad \text{and} \quad c(\Delta):=\int_0^\infty {k}\Big( e^{-(gt+\Delta_t)} \Big) dt<\infty \quad \textrm{a.s.}$$
For every $(t,\lambda) \in (\R^*_+)^2$, the solution $v_t$ of (\ref{eqnv}) {is non-decreasing on} $[0,t]$. Thus for all $s\in [0,t]$, $v_t(s,1,\Delta) \leq 1$, and
\[
 \begin{split}
  \psi_0(e^{-gs- \Delta_s}v_t(s,1,\Delta)) 
  &  \leq e^{-gs- \Delta_s}v_t(s,1,\Delta){k}(e^{-gs- \Delta_s}v_t(s,1,\Delta))\\
 & \leq e^{-gs- \Delta_s}v_t(s,1,\Delta){k}(e^{-gs- \Delta_s}) \qquad  \text{a.s}.
 \end{split}
 \]
 Then {(\ref{eqnv}) gives}
 $$ \frac{\partial}{\partial s} v_t(s,1,\Delta) \leq v_t(s,1,\Delta){k}(e^{-gs- \Delta_s}),
 $$ 
implying
$$-\ln(v_t(0,1,\Delta))\leq  \int_0^{t} {k}(e^{-gs- \Delta_s})ds\leq c(\Delta)<\infty \quad \text{ a.s.}$$
 Hence,  for every $t\geq 0$,  $v_t(0,1,\Delta)\geq \exp(-c(\Delta))>0$ and {conditionally} on $\Delta$ there exists a positive lower bound for $v_t(0,1,\Delta)$.
 Finally from (\ref{lap_exp}), 
$$\E_y[\exp\{- W\} \ \vert \Delta]=\exp \Big\{-y \underset{t \to \infty}{\lim}v_t(0,1,\Delta)  \Big\}<1$$
and $\P(W>0 \ \vert \ \Delta)>0 \ \text{a.s}.$ \\
Moreover, since $Y$ satisfies the branching property conditionally on $\Delta$, we can show  (see Lemma \ref{w0} in Section 6) that
$$\{W=0\} =\Big\{  \lim_{t\to\infty}Y_t=0\Big\} \qquad \text{ a.s.,}$$
which completes the proof.
 \end{proof}
We now derive a central limit theorem in the supercritical regime:
\begin{Cor}\label{TCL}
Assume that {$(gt+\Delta_t,t \geq 0)$} drifts to $+\infty$ and (\ref{xlogx}) is satisfied. Then, under the additional assumption
\begin{equation}\label{CLT}
\int_{(0,e^{-1}]\cup[e,\infty)} (\log {m})^2 \nu(\ud {m}) <\infty,
\end{equation}
 conditionally on $\{ W>0 \}$, 
$$\frac{\log (Y_t) -\textbf{m}t}{ \rho\sqrt{t}} \xrightarrow[t\to \infty]{d} {\mathcal{N}}(0,1),$$
where $\xrightarrow{d}$ means convergence in distribution,
\[
\textbf{m}:=g+\int_{\{|\log x |\ge 1\}} \log {m} \hspace{.05cm} \nu(\ud {m})<\infty,\qquad{\bf \rho}^2:=\int_{0}^\infty (\log {m})^2\nu(\ud {m})<\infty,
\]
 and ${\mathcal{N}}(0,1)$ denotes a centered Gaussian random variable with variance equals 1. 
\end{Cor}

\begin{proof}
We {use} the central limit theorem for the L\'evy process $(gt+\Delta_t, t\ge 0)$ under assumption (\ref{CLT})
of  Doney and Maller \cite{MR1922446}, see Theorem 3.5.  For simplicity,  the details are deferred to Section \ref{subTCL}. We then get
\begin{equation}\label{TLC1}
\frac{gt+\Delta_t-\textbf{m}t}{\rho \sqrt{t}}\xrightarrow[t\to \infty]{d}{\mathcal{N}}(0,1).
\end{equation}
From Corollary \ref{cor1} part $iii)$, under the event $\{ W>0 \}$, we get 
$$\log Y_t -(gt+\Delta_t)\xrightarrow[t\to \infty]{a.s.}\log  W  \in (-\infty,\infty),$$
 and we conclude  using (\ref{TLC1}).
\end{proof}

\section{Speed of extinction of  CSBP with catastrophes} 
\label{speed}
In this section, we {first} study   the particular case of the stable CSBP with growth  $g \in \R$.  Then, we derive a similar result for another class of CSBP's. 
\subsection{The stable case}
\label{stable}
We assume in this section that 
\begin{equation}\label{defstable}\psi(\lambda)={-}g\lambda+ c_+\lambda^{\beta+1},\end{equation}
for some $\beta \in (0,1]$,  $c_+>0$ and $g$ in $\R$.

 If $\beta=1$ (i.e. the Feller diffusion), we necessarily have  $\mu=0$ and the CSBP $Z$ follows the continuous diffusion 
$$Z_t=Z_0+\int_0^t gZ_s \ud s +\int_0^t\sqrt{2\sigma^2 Z_s} \ud B_s  , \quad t \geq 0.$$
In the case when  $\beta \in (0,1)$,  we necessarily have  $\sigma=0$ and the measure $\mu$ takes the form $\mu(\ud x)={c_+} (\beta+1) x^{-(2+\beta)}\ud x/\Gamma(1-\beta)$. In other words, the process possesses positive jumps with infinite intensity \cite{MR2299923}. Moreover,
$$Z_t=Z_0+\int_0^t gZ_s \ud s + \int_0^t Z_{s^-}^{1/(\beta+1)} \ud X_s{ , \quad t \geq 0},$$
where $X$ is a $(\beta+1)$-stable spectrally positive L\'evy process.\\

For the stable CSBP with catastrophes,  the backward differential equation (\ref{eqnv}) can be solved  and in particular, we get
\begin{Prop}  \label{prop}
For all $x_0> 0$ and $t\geq 0$:
\begin{equation} \label{pabs}
 \P_{x_0}(Y_t>0 \  \vert \  \Delta)=1-\exp\left\{- x_0 \left( c_+ {\beta} \int_0^t e^{-\beta({gs+\Delta_s})}\ud s \right)^{-1/\beta}\right\} \qquad \text{a.s}. 
\end{equation}
Moreover,
$$\P_{x_0}(\textrm{there exists } t>0, \  Y_t=0  \  \vert  \  \Delta )=1  \qquad \textrm{a.s.},$$
 if and only if the process {$(gt+\Delta_t, t\ge0)$}  does not drift to $+\infty$.
\end{Prop}
\begin{proof} %We solve equation (\ref{eqnv}) with  $\psi(\lambda)=g\lambda+c_+\lambda^{\beta+1}$. 
 Since $\psi_0(\lambda)=c_+\lambda^{\beta+1}$, a direct integration gives us
  \[
 v_{t}(u,\lambda, \Delta)=\left[c_+{\beta}\int_u^{t}e^{-\beta({gs+\Delta_s})} \ud s+\lambda^{-\beta}\right]^{-1/\beta},
 \]
which implies
\begin{equation} \label{trans_lap}
  \E_{x_0}\Big[e^{-\lambda \tilde{Z}_{t}}  \Big  \vert  \  \Delta  \Big]= \exp\left\{- x_0 \left( c_+{\beta} \int_0^{t} e^{-\beta({gs+\Delta_s})}\ud s+\lambda^{-\beta} \right)^{-1/\beta}\right\} \quad \text{a.s}. 
\end{equation}
Hence,   the absorption probability follows by letting  $\lambda$ tend to $\infty$
in  (\ref{trans_lap}). In other words,  
\[
 \P_{x_0}(Y_{t}=0  \  \vert  \  \Delta )= \exp\left\{- x_0 \left( c_+{\beta}\int_0^{t} e^{-\beta({gs+\Delta_s})}\ud s\right)^{-1/\beta}\right\}  \quad \text{a.s}.
\] 
Since $\P_{x_0}(\textrm{there exists } t\geq 0 : Y_{t}=0  \  \vert  \  \Delta )=\lim_{t\rightarrow \infty}\P_{x_0}(Y_{t}=0  \  \vert  \  \Delta )$ a.s., we deduce  
\[\P_{x_0}(\textrm{there exists } t\geq 0 : Y_{t}=0  \  \vert  \  \Delta )= \exp\left\{- x_0 \left( c_+{\beta}\int_0^{\infty} e^{-\beta({gs+\Delta_s})}\ud s\right)^{-1/\beta}\right\}  \quad \text{a.s}.\]
Finally, according to Theorem 1 in \cite{MR2178044},   
$ \int_0^\infty \exp\{-\beta({gs+\Delta_s})\}\ud s=\infty$  a.s. if and only if  the process $({gt+\Delta_t}, t\ge0)$ does not drift to $+\infty$. This completes  the proof.
\end{proof}

In what follows, we assume that the L\'evy process $\Delta$ admits some positive exponential moments, i.e. there exists $\lambda>0$ such that $\phi(\lambda)<\infty$.
We can then define $ \theta_{max}=\sup\{ \lambda>0,\,\phi(\lambda)<\infty \} \in (0,\infty]$ and we have
\be
 \label{condexp}
 \phi(\lambda):=\log\E[e^{\lambda \Delta_1}]=\int_0^{\infty}  (m^{\lambda}-1)\nu(dm) <\infty \qquad \text{ for } \lambda \in[ 0, \theta_{max}).
 \ee
We note that $\phi$ can be differentiated on the right in $0$  and also  in $1$ if $\theta_{max}>1$: 
$$\phi'(0):=\phi'(0+)=\int_0^{\infty}  \log(m) \nu(dm) \in (-\infty,\infty), \qquad \phi'(1)=\int_0^{\infty}  \log(m)m \nu(dm).$$

Recall that  $\Delta_t/t$ converges  to $\phi'(0)$ a.s. and  that $g+\phi'(0)$  is negative in the subcritical case. {Proposition \ref{prop}} then yields the asymptotic behavior of  the quenched  survival probability :
$$e^{-gt-\Delta_t} \P_{x_0}({ Y_t>0}  \vert  \  \Delta )\sim x_0\Big(c_+\beta\int_0^{t}e^{\beta({gt+\Delta_t-gs-\Delta_s})}ds\Big)^{-1/\beta} \quad (t\rightarrow \infty),$$
which converges in distribution to a positive finite limit  proportional to $x_0$. Then, 
$$\frac{1}{t} \log \P_{x_0}({ Y_t>0}  \vert  \  \Delta ) \rightarrow g+\phi'(0) \qquad (t\rightarrow \infty)$$
in probability.

Additional  work is required to get the asymptotic behavior of the annealed survival probability, for which four different regimes appear when  the process a.s. goes to zero:
%\marginpar{{\color{red} CS: j'ai enleve la dependance des constantes}}
\begin{Prop} \label{ppal_result}
We assume that $\nu$ satisfies (\ref{condh1}) and (\ref{CLT}), and that $\psi$ and $\phi$ satisfy (\ref{defstable}) and (\ref{condexp}) respectively.
\begin{enumerate}
\item[a/]
 If $\phi'(0)+g<0$ \emph{(subcritical case)} and $\theta_{max}>1$, then 
\begin{enumerate}
 \item[(i)] If $\phi'(1)+g<0$  \emph{(strongly subcritical regime)}, then there exists $c_1>0$ such that for every $x_0>0$,
$$ \P_{x_0}(Y_t>0) \sim c_1  x_0 e^{t(\phi(1)+g)}, \qquad \textrm{ as }\quad t\rightarrow \infty.$$

 \item[(ii)] If $\phi'(1)+g=0$ \emph{(intermediate subcritical regime)}, then there exists $c_2>0$ such that for every $x_0>0$,
$$\P_{x_0}(Y_t>0) \sim c_2 x_0 t^{-1/2}e^{t(\phi(1)+g)}, \qquad \textrm{ as }\quad t\rightarrow \infty.$$

 \item[(iii)] If $\phi'(1)+g>0$   \emph{(weakly subcritical regime)} and $\theta_{max}>\beta+1$, then for every $x_0>0$, there exists $c_3(x_0)>0$ such that
$$\P_{x_0}(Y_t>0) \sim c_3(x_0)t^{-3/2}e^{t(\phi(\tau)+g\tau)}, \qquad \textrm{ as }\quad t\rightarrow \infty, $$
where $\tau$ is the root of $\phi'+g$ on $]0,1[${: $\phi(\tau)+g\tau= \underset{0 < s < 1}{\min}\{\phi(s)+gs\}$. }

\end{enumerate}
\item[b/] If  $\phi'(0)+g=0$
 \emph{(critical case)} and  $\theta_{max}>\beta$, then for every $x_0>0$, there exists $c_4(x_0)>0$ such that 
$$\P_{x_0}(Y_t>0) \sim c_4(x_0)t^{-1/2}, \qquad \textrm{ as }\quad t\rightarrow \infty. $$

\end{enumerate}
\end{Prop}

\begin{proof}From  Proposition \ref{prop} we know that
\[
\P_{x_0}(Y_t>0)=1-\E \left[ \exp\left\{- x_0 \left( c_+ {\beta}\int_0^t e^{-\beta({gs+\Delta_s})}\ud s \right)^{-1/\beta}\right\} \right] = \E \left[ F\left( \int_0^t e^{-\beta K_s}\ud s \right) \right],
\]
where $F(x)=1-\exp\{-x_0(c_+{\beta} x)^{-1/\beta}\}$ and $ K_s=\Delta_s+gs.$ The {function} $F$ satisfies assumption (\ref{defF}) which is required in 
Theorem \ref{funcLev} (which is stated and proved in the next section). Hence Proposition \ref{ppal_result} follows from a direct application of this Theorem.
\end{proof}

In the case of CSBP's without catastrophes ($\nu=0$), the subcritical regime is reduced to  (i), and the critical case differs from b/, since  the asymptotic behavior is given by $1/t$. \\
In the strongly and intermediate subcritical cases $(i)$ and $(ii)$, $\E [Y_t] $  provides the  exponential decay factor of the survival 
probability which is given by  $\phi(1)+g$. 
Moreover the probability of non-extinction is proportional to the initial state $x_0$ of the population.
We refer to the proof of Lemma \ref{inter} and Section \ref{proofthp} for more details.\\
In the weakly subcritical case $(iii)$, the survival probability {decays} exponentially with rate  $\phi(\tau)+g\tau$, { which} is  strictly smaller than  $\phi(1)+g$.
In fact, as it appears in the proof of Theorem \ref{funcLev}, the quantity which determines   
the asymptotic behavior in all cases is $\E[\exp\{{\inf_{s \in [0,t]}(\Delta_s+gs)}\}]$. 
We also note  that $c_3$ and $c_4$ {may} not {be proportional to} $x_0$. We refer to \cite{MR2532207}
for a result in this vein for discrete branching processes in random environment. 

 More generally, the results stated above can be compared to the results which appear in the literature of discrete (time and space) branching processes in random environment (BPRE),
see e.g. \cite{MR1821473, MR1983172, Afanasyev2005}.
A BPRE $(X_n, n \in \N)$ is an integer valued branching process, specified by a sequence of generating functions $(f_n, n \in \N)$. Conditionally on the
 environment, individuals reproduce independently of each other and the offsprings of an individual at generation $n$ has generating function $f_n$.
We present briefly the results of Theorem 1.1 in \cite{MR1967796} and Theorems 1.1, 
1.2 and 1.3 in \cite{MR1983172}. To lighten the presentation, we do not specify here the moment conditions. \\
 In the {\it subcritical case}, i.e. when $\E[\log f_0'(1)]<0 $, we have the following three asymptotic regimes as $n$ increases,
$$\P(X_n>0) \sim ca_n, \qquad \textrm{as }\quad n\to\infty,$$
where  $c$ is a positive constant and $a_n$ is given by $$a_n=\E\Big[f_0'(1)\Big]^n, \quad a_n=n^{-1/2}\E\Big[f_0'(1)\Big]^n \quad \text{or} \quad a_n=n^{-3/2} \left(\underset{0 < s < 1}{\min}\E\Big[(f_0'(1))^s \Big]\right)^n,$$ 
when $\E[f_0'(1)\log f_0'(1)]$ is negative, zero or positive, respectively. \\
In the {\it critical case}, i.e. $ \E[\log f_0'(1)]=0 $, we have 
$$\P(X_n>0)\sim c n^{-1/2}, \qquad \textrm{as }\quad n\to\infty,$$
for some positive constant $c$. In the particular case when $\beta=1$, these results on BPRE and the approximation techniques  implemented in Section {\ref{demothm}} can be used to get  {Proposition \ref{ppal_result}}. We refer to {R}emarks \ref{th12GKV} and \ref{th11GKV} for more details. %As far as we know, the general case cannot be solved in the same way and requires some additional works, see  the next Sections.

Finally, in the continuous framework, such results have been established  for the Feller diffusion case, i.e. $\beta=1$,  whose drift varies following a Brownian motion (see \cite{BH}).
In other words the process $K$ is given by  a Brownian motion plus a drift. The techniques used by  the authors rely on an explicit formula for the Laplace transform of exponential functionals of Brownian motion which we cannot find in the literature for the case of  L\'evy processes. These results have been completed in the surpercritical regime in \cite{MR2868599}.

\subsection{Beyond the stable case.}
In this section, we  prove a similar result to {Proposition} \ref{ppal_result} for CSBP's with catastrophes in the case when the branching mechanism  $\psi_0$ is not stable. For technical reasons, we assume that  the Brownian coefficient is positive and  the associated L\'evy measure $\mu$ satisfies a second moment condition.

\begin{Cor}\label{cor} Assume that (\ref{condexp}) holds and
 $$\int_{(0,\infty)} {z}^2\mu(d{z})<\infty,  \qquad \sigma^2>0, \qquad  \int_{(0,\infty)}(\log {m})^2\nu(d{m})<\infty.$$
 \begin{enumerate}
  \item[a/]
{If} $\phi'(0)+g<0$  and $\theta_{max}>1$, {then} 
\begin{enumerate}
\item[(i)] If $\phi'(1)+g<0$, there exist $0<c_1\leq c_1'<\infty$ such that for every $x_0$,
$$ c_1 x_0  e^{t(\phi(1)+g)}\leq \P_{x_0}(Y_t>0) \leq c_1' x_0  e^{t(\phi(1)+g)}  \qquad  \text{for sufficiently large } t.$$
 \item[(ii)] If  $\phi'(1)+g=0$, there exist $0<c_2\leq c_2'<\infty$ such that for every $x_0$,
$$c_2 x_0  t^{-1/2}e^{t(\phi(1)+g)}\leq \P_{x_0}(Y_t>0) \leq c_2' x_0  t^{-1/2}e^{t(\phi(1)+g)}  \quad  \text{for sufficiently large }  t.$$
 \item[(iii)] If $\phi'(1)+g>0$ and $\theta_{max}>\beta+1$, for every $x_0$, there exist  $0<c_3(x_0)\leq c_3'(x_0)<\infty$ such that
$$ c_3(x_0)t^{-3/2}e^{t(\phi(\tau)+g\tau)}\leq\P_{x_0}(Y_t>0) \leq c_3'(x_0)t^{-3/2}e^{t(\phi(\tau)+g\tau)}  \quad {(t > 0)}, $$
where $\tau$ is the root of $\phi'+g$ on $]0,1[$.
\end{enumerate}

\item[b/] If $\phi'(0)+g=0$  and $\theta_{max}>\beta$, then  for every $x_0$, there exist  $0<c_4(x_0)<{c_4'(x_0)}<\infty$ such that
$$c_4(x_0)t^{-1/2}\leq \P_{x_0}(Y_t>0) \leq  c_4'(x_0)t^{-1/2}  \quad {(t > 0)}.$$
 \end{enumerate}
\end{Cor}

Note that the assumption $\sigma^2>0$ is only required for the upper bounds.

\begin{proof}
We recall that the branching mechanism associated with the CSBP $Z$ satisfies (\ref{defpsi}) for every $\lambda\geq 0$.
So for every $\lambda\geq 0$,
$$2\sigma^2\leq \psi''(\lambda) =  2\sigma^2 + \int_{(0,\infty)}{z}^2e^{-\lambda {z}}\mu(d {z}).$$
Since 
$c:=\int_0^{\infty} {z}^2 \mu(d{z})<\infty,$ 
$\psi''$ is continuous on $[0,\infty)$. By  Taylor-Lagrange's Theorem,
we get for every $\lambda \geq 0$, $\psi_{-}(\lambda)\leq \psi(\lambda)\leq \psi_{+}(\lambda),$
where
$$\psi_{-}(\lambda)=\lambda\psi'(0)+\sigma^2 \lambda^2\quad \text{and} \quad \psi_{+}(\lambda)= \lambda\psi'(0)+(\sigma^2+c/2) \lambda^2.$$

We first consider the case $\nu(0,\infty)<\infty$, so that $\Delta$ has a finite number of jumps on each compact interval a.s., and we also introduce the 
CSBP's with catastrophes $Y^-$ and $Y^+$ which have  the same 
catastrophes $\Delta$ as $Y$,  but with the characteristics $(g,\sigma^2,0)$ and $(g,\sigma^2+c/2,0)$, respectively. We denote 
$u_{-,t}$ and $u_{+,t}$ for  their respective Laplace exponent, in other words for all $(\lambda,t) \in \R_+^2$,
$$\E\Big[\exp\{-\lambda Y^-_t\}\Big]=\exp\{-u_{-,t}(\lambda)\}, \qquad \E\Big[\exp\{-\lambda Y^+_t\}\Big]=\exp\{-u_{+,t}(\lambda)\}.$$
Thus conditionally on $\Delta$, for every time $t$ such that $\Delta_{t}=\Delta_{t-}$,  we deduce, thanks to Theorem \ref{th1}, the following identities
 $$u_{-,t}'(\lambda)=-\psi_-(u_{-,t}), \qquad u_{+,t}'(\lambda)=-\psi_+(u_{+,t}), \qquad u_{t}'(\lambda)=-\psi(u_t).$$
Moreover for every $t$ such that $\theta_t=\exp\{\Delta_{t}-\Delta_{t-}\}{\neq 1}$, 
$$\frac{u_{-,t}(\lambda)}{u_{-,t-}(\lambda)}=\frac{u_{t}(\lambda)}{u_{t-}(\lambda)}=\frac{u_{+,t}(\lambda)}{u_{+,t-}(\lambda)}=\theta_t,$$
and $u_{-,0}(\lambda)=u_{0}(\lambda)=u_{+,0}(\lambda)=\lambda.$
So for all $t,\lambda$, we have
$$u_{+,t}(\lambda) \leq u(t,\lambda) \leq u_{-,t}(\lambda).$$
The extension   of the above  inequality to the case $\nu(0,\infty) \in [0 , \infty]$ can be achieved by successive approximations. We defer the technical details to Section \ref{details}. \\
Having {into} account that  the above inequality holds in general, we deduce,  taking $\lambda\rightarrow \infty$, that 
$$\P(Y^{+}_t >0) \leq \P(Y_t>0)\leq \P(Y^{-}_t >0). $$
The result then follows from  the asymptotic behavior of $\P(Y^{-}_t >0)$ and $\P(Y^{+}_t >0)$, which are inherited from  Proposition \ref{ppal_result}. 
\end{proof}

\section{Local limit theorem for some functionals of L\'evy processes}\label{demothm}
\label{proof}

 We proved in Proposition \ref{prop} that the probability that a stable CSBP with catastrophes becomes extinct at time $t$ equals the expectation of a functional of a L\'evy process. 
We now prove the key result of the paper. It deals with the asymptotic behavior of the mean of some L\'evy functionals.\\
More precisely, we are interested in the  asymptotic behavior at infinity of
$$a_F(t):=\E\left[F\left(\int_0^{t} \exp\{-\beta K_s\}\ud s\right)\right],$$
where $K$ is a L\'evy process with bounded variation paths and $F$ belongs to a 
particular class of functions on $\R_+$. We will focus on functions which decrease  polynomially at infinity  (with exponent 
$-1/\beta$). The motivations come from the previous section. {In particular, the Proposition \ref{ppal_result} is a direct application of Theorem \ref{funcLev}}.% for their use in the study of asymptotic probability of branching processes with catastrophes.

Thus, we consider a L\'evy process $K=(K_t,t\ge 0)$ of the form
\begin{equation}\label{defK}
K_t=\gamma t+\sigma^{(+)}_t-\sigma^{(-)}_t, \qquad t\ge 0,
\end{equation}
where $\gamma$ is a real constant, $\sigma^{(+)}$ and $\sigma^{(-)}$ are two independent {pure jump} subordinators. 
We denote  by $\Pi$, $\Pi^{(+)}$ and $\Pi^{(-)}$  the associated  L\'evy measures of $K$, $\sigma^{(+)}$ and $\sigma^{(-)}$, respectively.  We also define the Laplace exponents of $K$, $\sigma^{(+)}$ and $\sigma^{(-)}$ by 
\begin{equation}\label{defphi} \phi_K(\lambda)=\log\E\Big[e^{\lambda K_1}\Big], \quad {\phi_K^+}(\lambda)=\log \E\Big[e^{\lambda \sigma_1^{(+)}}\Big] \quad\text{ and }\quad {\phi_K^-}(\lambda)=\log \E\Big[e^{-\lambda \sigma_1^{(-)}}\Big], \end{equation}
 and assume that 
\begin{equation}\label{theta}
 \theta_{max} = {\sup} \left\{ \lambda \in \R^+, \int_{[1,\infty)} e^{\lambda x}  \Pi^{(+)}(\ud x)<\infty \right\}>0.
\end{equation}
From the  L\'evy-Khintchine formula, we deduce
\[
\phi_K(\lambda)=\gamma \lambda+\int_{(0,\infty)}\Big(e^{\lambda x}-1\Big)\Pi^{(+)}(\ud x)+\int_{(0,\infty)}\Big(e^{-\lambda x}-1\Big)\Pi^{(-)}(\ud x).
\]
Finally, we assume that $\E[K^2_1]<\infty$, which is equivalent to
\begin{equation}\label{mmt2}
 \int_{(-\infty,\infty)}x^2 \Pi(dx)<\infty.
\end{equation}

\begin{Thm}
\label{funcLev}  Assume that (\ref{defK}), (\ref{theta})  and (\ref{mmt2}) hold.  Let $\beta \in (0,1]$ and $F$ be a positive non increasing function such that for $x\ge 0$ 
\begin{equation}
\label{defF} F(x)=C_F (x+1)^{-1/\beta}\Big[ 1+(1+x)^{-\varsigma}h(x) \Big] ,
\end{equation}
where $\varsigma\geq 1$, $C_F$ is a positive constant, and $h$ is a Lipschitz function which is bounded.
%  Then we have the four following regimes
\begin{enumerate}
 \item[a/]  If  $\phi_K'(0)<0$ 
\begin{itemize}
\item[(i)] If $\theta_{max}>1$ and $\phi_K'(1)<0$, there exists a positive constant $c_1$ such that 
$$ a_F(t) \sim c_1 e^{t\phi_K(1)}, \quad \textrm{ as }\quad t\rightarrow \infty.$$
\item[(ii)] If $\theta_{max}>1$ and $\phi_K'(1)=0$, there exists a positive constant $c_2$ such that 
$$ a_F(t) \sim c_2 t^{-1/2} e^{t\phi_K(1)}, \quad \textrm{ as }\quad t\rightarrow \infty.$$
\item[(iii)] If $\theta_{max} > \beta+1$ and $\phi_K'(1)>0$, there exists a positive constant $c_3$ such that 
$$ a_F(t) \sim c_3 t^{-3/2} e^{t\phi_K(\tau)}, \quad \textrm{ as }\quad t\rightarrow \infty,$$
where $\tau$ is the root of $\phi_K'$ on $]0,1[$.
\end{itemize}
\item[b/] If $\theta_{max}>\beta$ and $\phi_K'(0)= 0$, there exists a positive constant $c_4$ such that 
$$ a_F(t) \sim c_4 t^{-1/2} , \quad \textrm{ as }\quad t\rightarrow \infty.$$
\end{enumerate}
\end{Thm}

 This  result generalizes Lemma 4.7 in Carmona et al. \cite{MR1648657}  in the case when the process $K$ has bounded variation paths.  More precisely, the authors in  \cite{MR1648657} only provide a precise asymptotic behavior in the case  when $\phi_K'(1)< 0$. 

The assumption on the behavior of $F$ as $x\rightarrow \infty$ is finely used to get the asymptotic behavior of $a_F(t)$. Lemma \ref{lemtecaussi(quel_nom_original...)} gives the  properties of $F$ which are required in the proof.

The strongly subcritical case (case (i)) is proved using a continuous time change of measure (see Section \ref{proofthp}). 
For the remaining  cases, we  divide the proof in three steps. The first one (see Lemma \ref{discret}) consists in  
discretizing the exponential functional $\int_0^{t} \exp(-\beta K_s)\ud s$ using the random variables 
\begin{equation}\label{Apq}
 A_{p,q}=\underset{i=0}{\overset{p}{\sum}}\exp\{-\beta K_{i/q}\}=  \sum_{i=0}^p \prod_{j=0}^{i-1} \exp\Big\{-\beta \big(K_{{(j+1)}/q}-K_{j/q}\big)\Big\} \quad   ((p,q) \in \N\times \N^*).
\end{equation}
Secondly (see Lemmas \ref{inter}, \ref{weak} and \ref{crit}), we study the asymptotic behavior of the discretized expectation 
\begin{equation}\label{Fpq}F_{p,q}:= \E \Big[F \Big( A_{p,q}/q\Big) \Big]  \quad (q \in \N^*), \end{equation}
when $p$ goes to infinity. This  step relies
 on Theorem 2.1 in \cite{MR1821473}, which  is a limit theorem for random walks on an affine group  
 and  generalizes theorems A and B in
 \cite{MR1443957}.\\
Finally (see Sections \ref{conv_const} and \ref{proofthp}), we prove that the limit of  $F_{\lfloor qt\rfloor,q}$, when $q\to \infty$, and  $a_F(t)$ both 
have the same asymptotic behavior when $t$ goes to infinity.

\subsection{Discretization of the L\'evy process}

The following result, which follows from the property of independent and stationary increments of the process $K$, allows us to concentrate on $A_{p,q},$ 
which has been defined in (\ref{Apq}).

\begin{Lem}\label{discret}
Let $t \geq 1$ and $q \in \N^*$. Then 
$$\frac{1}{q}e^{-\beta(|{\gamma}|/q+\sigma^{(+)}_{1/q})} A^{(1)}_{\lfloor qt \rfloor-1,q}\leq  \int_0^t e^{-\beta K_s}\ud s  \leq \frac{1}{q}e^{\beta (|{\gamma}|/q+\sigma^{(-)}_{1/q})}A^{(2)}_{\lfloor qt \rfloor,q} ,$$
where for every $(p,q) \in \N \times \N^*$,  
$\sigma^{(+)}_{1/q}$ (resp $\sigma^{(-)}_{1/q}$) is independent of $A^{(1)}_{p,q}$ (resp $A^{(2)}_{p,q}$) and $$A_{p,q}\overset{(d)}{=}A^{(1)}_{p,q}\overset{(d)}{=}A^{(2)}_{p,q}.$$
\end{Lem}

\begin{proof}
 Let $(p,q)$ be in $\N \times \N^*$ and $s \in[p/q,(p+1)/q]$. Then 
\begin{equation}\label{majmin}
 K_s \leq K_{p/q} +|\gamma|/q+[\sigma^{(+)}_{(p+1)/q}-\sigma^{(+)}_{p/q}]\quad \text{and} \quad
 K_s\geq K_{p/q}- |\gamma|/q-[\sigma^{(-)}_{(p+1)/q}-\sigma^{(-)}_{p/q}] .
\end{equation} 
Now introduce 
$$  K_{p/q}^{(1)}=K_{p/q}+[\sigma^{(+)}_{(p+1)/q}-\sigma^{(+)}_{p/q}]-\sigma^{(+)}_{1/q}=\gamma p/q+[\sigma^{(+)}_{(p+1)/q}-\sigma^{(+)}_{1/q}]-\sigma^{(-)}_{p/q},$$
and
$$ K_{p/q}^{(2)}=K_{p/q}-[\sigma^{(-)}_{(p+1)/q}-\sigma^{(-)}_{p/q}]+\sigma^{(-)}_{1/q}=\gamma p/q+\sigma^{(+)}_{p/q}-[\sigma^{(-)}_{(p+1)/q}-\sigma^{(-)}_{1/q}]. $$
Then, we have for all $(p,q) \in \N \times \N^*$ 
$$ (K_0,K_{1/q},...,K_{p/q})\overset{(d)}{=}(K_0^{(1)},K_{1/q}^{(1)},...,K_{p/q}^{(1)})\overset{(d)}{=}(K_0^{(2)},K_{1/q}^{(2)},...,K_{p/q}^{(2)}). $$
Moreover, the random vector $(K_0^{(1)},K_{1/q}^{(1)},...,K_{p/q}^{(1)})$ is independent of $\sigma^{(+)}_{1/q}$ and $(K_0^{(2)},K_{1/q}^{(2)},...,K_{p/q}^{(2)})$ is 
independent of $\sigma^{(-)}_{1/q}$. Finally, the definition of $$A_{p,q}^{(i)}=\underset{i=0}{\overset{p}{\sum}}\exp\{-\beta K_{i/q}^{(i)}\}$$ for $i \in \{1,2\}$ and the inequalities in (\ref{majmin}) complete the proof.
\end{proof}

\subsection{Asymptotical behavior  of the discretized process}

First, we  recall Theorem 2.1 of \cite{MR1821473} in the case where the test functions do not vanish.  This is the key result  to obtain the asymptotic behavior  of the discretized process.

\begin{Thm}[Giuvarc'h, Liu 01]
\label{th3} 
 Let $(a_n,b_n)_{n \geq 0}$ be a $(\R_+^*)^2$-valued sequence of iid random variables such that $\E[\log(a_0)]=0$. Assume that $b_0/(1-a_0)$ is not constant a.s. and define $A_0=1$, $A_n=\prod_{k=0}^{n-1}a_k$ and {$B_n=\sum_{k=0}^{n-1}A_kb_k$}, for $n\geq 1$. Let $ \eta, \kappa, \xi$ be three positive numbers such that $\kappa <\xi$, and $\tilde{\phi}$ and $\tilde{\psi}$ be two positive continuous functions on $\R_+$  such that they do not vanish and for a constant $C>0$ and for every $a>0$, $b\geq 0$, $b'\geq 0$, we have
$$ \tilde{\phi}(a)\leq C a^\kappa, \quad \tilde{\psi}(b)\leq \frac{C}{(1+b)^\xi}, \quad\text{ and }\quad |\tilde{\psi}(b)-\tilde{\psi}(b')|\leq C |b-b'|^\eta. $$
Moreover, assume that 
$$ \E\big[a_0^\kappa\big]<\infty,\quad \E\big[a_0^{-\eta}\big]<\infty,\quad \E\big[b_0^{\eta}\big]<\infty\quad\text{ and }\quad \E\big[a_0^{-\eta}b_0^{-\eta}\big]<\infty. $$
Then there exist two positive constants $c(\tilde{\phi},\tilde{\psi})$ and $c(\tilde{\psi})$ such that 
$$ \underset{n \to \infty}{\lim}n^{3/2}\E \left[ \tilde{\phi}(A_n)\tilde{\psi}(B_n) \right]=c(\tilde{\phi},\tilde{\psi})
\qquad \text{ and }\qquad \underset{n \to \infty}{\lim}n^{1/2}\E \left[ \tilde{\psi}(B_n) \right]=c(\tilde{\psi}).  $$
\end{Thm}

Let us now state a technical lemma on the tail of function $F$, useful to get the asymptotical behaviour of the disretized process. Its proof is deferred 
to Section \ref{techlem} for the convenience of the reader.

 \begin{Lem} \label{lemtecaussi(quel_nom_original...)}
 Assume that $F$ satisfies (\ref{defF}). Then there exist two positive  finite constants $\eta$ and $M$ such that for all  $(x,y)$ in $\R_+^2$ and $\varepsilon $ in $[0,\eta]$,
\bea
\label{majepsilon}
 \Big| F(x)-C_F x^{-1/\beta} \Big| &\leq & Mx^{-(1+\varepsilon)/\beta}, \\
\label{controle}
 \Big| F(x)-F(y) \Big| &\leq &M\Big| x^{-1/\beta}-y^{-1/\beta} \Big|.
\eea
\end{Lem}

Recall the definitions of $A_{p,q}$ and $F_{p,q}$ in (\ref{Apq}) and (\ref{Fpq}), respectively. The three following lemmas study the asymptotic behavior of $F_{p,q}$ and the {mean} value of $(A_{p,q}/q)^{-1/\beta}$ in the regimes of (ii), (iii) and b/.

\begin{Lem}\label{inter}
 Assume that $|\phi'_K(0+)|<\infty$,  $\theta_{max}>1$ and $ \phi'_K(1)=0 $. Then there exists a positive and finite constant $c_2(q)$ such that, 
\begin{equation}\label{defc2}
 F_{p,q} {\sim}   C_F c_2(q)  (p/q)^{-1/2}e^{(p/q)\phi_K(1)}, \quad \textrm{as }\quad p\to\infty,
\end{equation}

and 
\begin{equation}\label{*}
 \E\left[ \left( A_{p,q}/ q\right)^{-1/\beta} \right] {\sim}c_2(q)  (p/q)^{-1/2}e^{(p/q)\phi_K(1)}, \quad \textrm{as }\quad p\to\infty.
\end{equation}
\end{Lem}

\begin{proof} 
Let {us} introduce the exponential change of measure known as the Escheer transform
\begin{equation}\label{escheer}
 \frac{\ud \P^{(\lambda)}}{\ud \P}\bigg|_{\mathcal{F}_t}= e^{\lambda K_t-\phi_K(\lambda)t}  \qquad \textrm{for }\lambda \in [0 ,\theta_{max}),
 \end{equation}
where $(\mathcal{F}_t)_{t\ge 0}$ is the natural filtration generated by $K$ which is naturally completed.

The following equality in law 
$$ A_{p,q}=e^{-\beta K_{p/q}}\Big( \underset{i=0}{\overset{p}{\sum}} e^{\beta (K_{p/q}-K_{i/q})}\Big)\overset{(d)}{=}e^{-\beta K_{p/q}}\Big( \underset{i=0}{\overset{p}{\sum}} e^{\beta K_{i/q}}\Big) ,$$
 leads to 
$ e^{-(p/q)\phi_K(1)} \E \left[ A_{p,q}^{-1/\beta}\right]=\E^{(1)} \left[\tilde{A}_{p,q}^{-1/\beta}\right],$
where $ \tilde{A}_{p,q}={\sum}_{i=0}^p e^{\beta K_{i/q}}$.  Let $\varepsilon>0$ {be} such that (\ref{majepsilon}) holds and observe  that $\tilde{A}_{p,q}\geq 1$ {a.s.} for every $(p,q)$ in $\N \times \N^*$. Thus,
\[ 
\E^{(1)} \left[ \tilde{A}_{p,q}^{-{(1+\varepsilon)}/{\beta}}\right] \leq   \E^{(1)} \left[\tilde{A}_{p,q}^{-{1}/{\beta}}\right] \leq  \E^{(1)} \left[ \underset{i \in [0,p]\cap \mathbb{N}}{\inf}e^{-K_{i/q}} \right].
\]
Since $\phi_K^\prime(1)=0$ and $\E[K^2_{1/q}]<\infty$, Theorem $A$ in \cite{koz} implies 
\[\E^{(1)} \left[ \underset{i \in [0,p]\cap \mathbb{N}}{\inf}e^{-K_{i/q}} \right] \sim \hat{C}_q (p/q)^{-1/2}, \quad \textrm{as }\quad p\to\infty, \]
where $\hat{C}{_q}$ is a finite positive  constant.
We define  for $z \geq 1$,
$$ D_q(z,p)=(p/q)^{1/2}\E^{(1)} \left[ \tilde{A}_{p,q}^{-{z}/{\beta}} \right]. $$
Moreover, we note that there exists $p_0 \in \N$ such that for $p\ge p_0$, $ D_q({1},p)  \leq 2\hat{C}_q$.

 Our aim is to prove that $D_q(1,p)$ converges, {as $p$ increases}, to a finite positive constant $d_2(q)$. Then, we introduce an arbitrary $x \in (0,(C_F/M)^{1/\varepsilon}q^{-1/\beta})$ and apply Theorem 
\ref{th3} with
$$  \tilde{\psi}(z)=F(z),\quad {\tilde{\phi}(z)=z^{1/(2\beta)}}, \qquad (\eta,\kappa,\xi)=(1,1/(2\beta),1/\beta). $$
Observe that $F$ is a Lipschitz function and that under the probability measure $\P^{(1)}$, 
$(a_n,b_n)_{n \geq 0}=(\exp(\beta (K_{(n+1)/q}-K_{n/q})),x^{-\beta}q^{-1})_{n \geq 0}$ is an i.i.d. sequence of random 
variables with $\E^{(1)} [\log(a_0)] = 0$, since $\phi_K'(1)=0$. Moreover,
a simple computation gives
$$\E^{(1)}[a_0^{-1}]=e^{(\phi_K(1-\beta)-\phi_K(1) )/q}<\infty,$$
so that  the moment conditions of Theorem \ref{th3}  are satisfied. We apply the result with
$$ B_n=q^{-1}{x^{-\beta}}\underset{i=0}{\overset{n-1}{\sum}} e^{\beta K_{i/q}}, \quad n \in \N^*$$
and we get the existence of a positive  finite real number $b(q,x)$ such that
$$ (p/q)^{1/2} \E^{(1)} \left[ F\Big(x^{-\beta} \tilde{A}_{p,q}/q \Big) \right]\rightarrow b(q,x) , \quad \textrm{as }\quad p\to\infty. $$
Taking expectation in  (\ref{majepsilon}) yields
\begin{equation}\label{ineg} \left\vert (p/q)^{1/2} \E^{(1)} \left[ F\Big(x^{-\beta} \tilde{A}_{p,q}/q \Big) \right]- C_F x {q^{1/\beta}} D{_q}(1,p)\right\vert\leq M x^{1+\varepsilon}{q^{(1+\varepsilon)/\beta}}D{_q}(1+\varepsilon,p).\end{equation}
Defining
$\underline{D}_q:=\liminf_{p\rightarrow \infty}D_q(1,p)$ and $\overline{D}_q:=\limsup_{p\rightarrow \infty}D_q(1,p)$, 
we combine the two last dispalys to get
$$ C_F x {q^{1/\beta}}\overline{D}{_q}\leq b(q,x)+M x^{1+\varepsilon}{q^{(1+\varepsilon)/\beta}}\limsup_{p\rightarrow\infty} D{_q}(1+\varepsilon,p), $$
and
$$ C_F x {q^{1/\beta}}\underline{D}{_q}\geq b(q,x)-M x^{1+\varepsilon}{q^{(1+\varepsilon)/\beta}} \limsup_{p\rightarrow\infty}D{_q}(1+\varepsilon,p). $$
Adding that $D_q(z,p)$ is non-increasing with respect to $z$,  $D_q(1+\varepsilon,p)\leq D_q(1,p)\leq 2\hat{C}_q$ for every $p\geq p_0$ and
$$ \overline{D}{_q}-\underline{D}{_q}\leq {\frac{4M\hat{C}_qx^{\varepsilon}q^{\varepsilon/\beta}}{C_F}}. $$
Finally, letting  $x\rightarrow 0$, we get that $D{_q}(1,p)$ converges to a  finite constant $d_2(q)$. Moreover, from \eqref{ineg}, we get for every integer $p$:
$$  (C_Fx{q^{1/\beta}}+M x^{1+\varepsilon}{q^{(1+\varepsilon)/\beta}})D{_q}(1,p) \geq (p/q)^{1/2} \E^{(1)} \left[ F\Big(x^{-\beta} \tilde{A}_{p,q}/q \Big) \right].$$
Letting $p \to \infty$, we get that $d_2(q)$ is positive, which gives (\ref{*}).

Now, using (\ref{majepsilon}), we get
$$\E \Big| F_{p,q}- C_F \left( A_{p,q}/ q\right)^{-1/\beta} \Big| \leq \E\left[ \left( A_{p,q}/ q\right)^{-(1+\varepsilon)/\beta} \right],$$
so the asymptotic behavior in (\ref{defc2}) will be proved as soon as we show that
$$ \E \Big[  A_{p,q}^{-{(1+\varepsilon)}/{\beta}} \Big] = o \left( \E \left[  A_{p,q}^{-{1}/{\beta}}\right]\right), \quad \textrm{as }\quad p\to\infty. $$

From the Escheer transform (\ref{escheer}), with $\lambda=1+\varepsilon$, and the {independence of the increments} of $K$, we have 
\begin{eqnarray*}
 \E \Big[  A_{p,q}^{-(1+\varepsilon)/\beta} \Big] &   = & e^{(p/q)\phi_K(1)} \E^{(1)}\Big[ \Big(\underset{i=0}{\overset{p}{\sum}}e^{-\beta  K_{i/q}}\Big)^{-\varepsilon/\beta} \Big(\underset{i=0}{\overset{p}{\sum}}e^{\beta(K_{p/q}-K_{i/q})}\Big)^{-1/\beta} \Big] \\
& \leq& e^{(p/q)\phi_K(1)}  \E^{(1)}\Big[\underset{0 \leq i \leq \lfloor p/3 \rfloor}{\inf}e^{\varepsilon K_{i/q}} \underset{{\lfloor 2p/3 \rfloor} \leq j \leq p}{\inf}e^{-(K_{p/q}-K_{j/q})} \Big]\\
 &= & e^{(p/q)\phi_K(1)}\E^{(1)}\Big[\underset{0 \leq i \leq \lfloor p/3 \rfloor}{\inf}e^{\varepsilon K_{i/q}} \Big] \E^{(1)}\Big[\underset{0\leq j \leq \lfloor p/3 \rfloor   }{\inf}e^{-K_{j/q}} \Big].
\end{eqnarray*}
 Using (\ref{mmt2}), we observe that  $\E^{(1)}[K_{1/q}]=0$ and
 $\E^{(1)}[K_{1/q}^2]<\infty$. We can then apply Theorem $A$ in \cite{koz} to the random walks $(-K_{i/q})_{i\geq 1}$ and $(\varepsilon K_{i/q})_{i\geq 1}$. Therefore, there exists ${C(q)}>0$ such that 
$$ \E \Big[  A_{p,q}^{-{(1+\varepsilon)}/{\beta}} \Big] \leq (C(q)/p)e^{(p/q)\phi_K(1)} = o \left( \E \left[  A_{p,q}^{-{1}/{\beta}}\right]\right), \quad \textrm{as }\quad p\to\infty. $$
Taking $c_2(q)=d_2(q)q^{1/\beta}$ leads to the result.
\end{proof}

\begin{Rque}\label{th12GKV}
In the particular case when $\beta=1$, it is enough to apply Theorem 1.2 in \cite{MR1983172} to a geometric BPRE $(X_n, n\geq 0)$ whose p.g.f's satisfy 
 \begin{equation*}\label{casgeo}  f_n(s)=\underset{k=0}{\overset{\infty}{\sum}}p_n q_n^k s^k = \frac{p_n}{1-q_n s}, \end{equation*}
with 
$ 1/p_n=1+\exp\left\{\beta\left(K_{{(n+1)}/{q}}-K_{{n}/{q}}\right)\right\}$, and $q_n=1-p_n.$ 
Using
$\E[A_{p,q}^{-1}]=\P(X_p>0)$ and  $ \log f_0'(1)= K_{1/q} ,$
allows to get the asymptotic behavior of $\E[A_{p,q}^{-1}]$ from the speed of extinction of BPRE in the case of geometric
reproduction law (with the extra assumption $\phi_K(2)<\infty$). 
\end{Rque}

 Recall that $\tau$ is the root of $\phi'_K$ on $]0,1[$, i.e. $ \phi_K(\tau)= {\min}_{0 < s < 1}\phi_K(s).$ 

\begin{Lem}\label{weak}
Assume that $\phi'_K(0)<0$,  $ \phi'_K(1)>0 $ and ${\theta_{max}>\beta+1}$. Then there exist two positive constants $d(q)$ and $c_3(q)$ such that
\begin{eqnarray}\label{defc3}
 F_{p,q} &\sim&  c_3(q)(p/q)^{-3/2}e^{(p/q)\phi_K(\tau)}, \quad \textrm{as }\quad p\to\infty, \end{eqnarray}
 and 
\begin{eqnarray}\label{defd}\E\left[  (A_{p,q}/q)^{-1/\beta} \right] &\sim& d(q) (p/q)^{-3/2}e^{(p/q)\phi_K(\tau)}, \quad \textrm{as }\quad p\to\infty.
\end{eqnarray}
\end{Lem}
\begin{proof} {First}  we apply Theorem \ref{th3} where, for $z\ge 0$,
$$  \tilde{\psi}(z)=F(z),\qquad \tilde{\phi}(z)=z^{\tau/\beta}, \qquad (\eta,\kappa,\xi)=(1,\tau/\beta,1/\beta). $$
Again $F$ is a Lipschitz function, and under the probability measure $\P^{(\tau)}$, 
$(a_n,b_n)_{n \geq 0}=(\exp(-\beta(K_{(n+1)/q}-K_{n/q})),q^{-1})_{n \geq 0},$ is an i.i.d. sequence of random variables such that $\E^{(\tau)} [\log(a_0)] = 0$, {since $\phi'_K(\tau)=0$}. The  moment conditions   
 $$ \E^{(\tau)}\big[a_0^{\tau/\beta}\big]=e^{-\phi_K(\tau)/q}<\infty\qquad
\textrm{and}\qquad \E^{(\tau)}\big[a_0^{-1}\big]=e^{(\phi_K(\beta+\tau)-\phi_K(\tau) )/q}<\infty,$$
enable us to apply  Theorem \ref{th3}.
In this case,
$$ B_n=q^{-1}\underset{i=0}{\overset{n-1}{\sum}} e^{-\beta K_{i/q}}, \quad n \in \N^*{.} $$
 Then there exists $c_3(q)>0$ such that 
\begin{eqnarray*}
 \E \left[F(A_{p,q}/q)\right] e^{-(p/q)\phi_K(\tau) }=\E^{(\tau)} \left[F(A_{p,q}/q) e^{-\tau K_{p/q}}\right]\sim& c_3(q)(p/q)^{-3/2},
\end{eqnarray*}
as $p\to\infty$. This gives (\ref{defc3}).

To prove 
 $$  \E \left[  (A_{p,q}/q)^{-1/\beta}\right]\sim d(q)(p/q)^{-3/2}e^{\frac{p}{q}\phi_K(\tau)}, \quad \textrm{as }\quad p\to\infty$$
for  $d(q)>0$, we follow the same arguments as those used in  the proof of Lemma \ref{inter}. In other words, 
we define for $z \geq 1$,
$$ D_q(z,p)=(p/q)^{3/2}e^{-(p/q)\phi_K(\tau)}\E \left[ A_{p,q}^{-z/\beta} \right], $$
which is non-increasing with respect to $z$. We obtain the same type of inequalities as in  Lemma \ref{inter}, for the random variable $A$ instead of $\tilde{A}$.

Again we take $\varepsilon>0$ such that (\ref{majepsilon}) holds. Then Lemma 7 in \cite{MR1633937} {yields the existence of}  $C_q>0$ such that for $p$ large enough,
\[ 
\E \left[ A_{p,q}^{-(1+\varepsilon)/\beta}\right] \leq  \E \left[ A_{p,q}^{-1/\beta}\right]\leq   \E \left[ \underset{i \in [0,p]\cap \N}{\inf}e^{-K_{i/q}} \right]\sim 
C_q ({p}/{q})^{-3/2} e^{(p/q)\phi_K(\tau)}.
\]
Finally, we  use Theorem \ref{th3} to get
$0 < \liminf_{n\rightarrow \infty} D_q(1,n)=\limsup_{n\rightarrow \infty} D_q(1,n) <\infty, $
which completes the proof. 
\end{proof}

\begin{Lem}\label{crit}
Assume that $\phi'_K(0)=0$  and  {$\theta_{max}>\beta$}. Then there exist two positive constants $b(q)$ and $c_4(q)$ such that 
\begin{equation}\label{defc4}
 F_{p,q} \sim c_4(q)(p/q)^{-1/2},  \quad \textrm{as }\quad p\to\infty,
 \end{equation}
 and
\begin{equation}\label{defb}
\E\left[  (A_{p,q}/q)^{-1/\beta} \right] \sim b(q)  (p/q)^{-1/2},  \quad \textrm{as }\quad p\to\infty.
\end{equation}
\end{Lem}
\begin{proof} The proof is  almost the same as the proof of  Lemma \ref{weak}. We first apply  Theorem \ref{th3} to the same function $\tilde{\psi}$ and
 sequence $(a_n,b_n)_{n\geq 0}$ defined in Lemma \ref{weak} but under the probability measure $\P$ instead of $\P^{(\tau)}$. Then, we get 
\begin{eqnarray*}
  \E \Big[F ( A_{p,q}/ q) \Big] \sim  c_4(q)(p/q)^{-1/2}, \quad \textrm{as }\quad p\to\infty.
\end{eqnarray*}
 Now, we define for $z \geq 1$,
$$ D_q(z,p)=(p/q)^{1/2}\E \Big[ A_{p,q}^{-z/\beta} \Big], $$
and from Theorem A in \cite{koz} and Theorem \ref{th3}, we obtain that  $D_q(1,p)$ has a positive finite limit when $p$ goes to infinity.
\end{proof}

\subsection{From the discretized process to the continuous process} \label{conv_const} 

 Up to now, the asymptotic behavior of the processes {was depending} on the step size $1/q$. By letting  $q$ tend to infinity, we obtain our results in continuous time. 
 To do this we shall use several times a technical Lemma on limits of sequences.
\begin{Lem}\label{lemtec}
 Assume that the non-negative sequences $(a_{n,q})_{(n,q) \in \N^2}$, $(a'_{n,q})_{(n,q) \in \N^2}$ and $(b_{n})_{n \in \N}$ satisfy for every $(n,q) \in \N^2$:
$$ a_{n,q}\leq b_n \leq a'_{n,q}, $$
and that there exist  three sequences $(a(q))_{q \in \N}$, $(c^-(q))_{q \in \N}$ and $(c^+(q)_{q \in \N}$ such that
$$ \underset{n \to \infty}{\lim}a_{n,q}=c^-(q)a(q), \quad \underset{n \to \infty}{\lim}a'_{n,q}=c^+(q)a(q), \quad \text{and} \quad \underset{q \to \infty}{\lim}c^-(q)=\underset{q \to \infty}{\lim}c^+(q)=1. $$
Then there exists a non-negative constant $a$ such that 
$$ \underset{q \to \infty}{\lim}a(q)=\underset{n \to \infty}{\lim}b_{n}=a .$$
\end{Lem}
\begin{proof}
 From our assumptions, it is clear that for every $q \in \N$
$$ \limsup_{n\to \infty} b_n \leq c^+(q)a(q) \quad \text{and} \quad  c^-(q)a(q) \leq\liminf_{n\to \infty} b_n. $$
Then letting $q$ go to infinity, we obtain
$$ \limsup_{n\to \infty} b_n \leq \liminf_{q\to \infty} a(q) \quad \text{and} \quad  \limsup_{q\to \infty} a(q) \leq\liminf _{n\to \infty}b_n, $$
which ends the proof.
\end{proof}

Recalling  the notations  (\ref{defc2}) to (\ref{defb}), we prove the following limits :

\begin{Lem}\label{conv_cons}
There exist five finite positive constants $b$, $d$, $c_2$, $c_3$ and $c_4$ such that
\begin{equation}\label{conv} 
 (b(q),d(q),c_2(q),c_3(q),c_4(q))\longrightarrow (b,d,c_2,c_3, c_4 ),  \quad \textrm{as }\quad q\to\infty.\end{equation} 
\end{Lem}

\begin{proof} First we prove the convergence of $d(q)$. From Lemma \ref{discret}, we know that for every $n \in \N^*$
\begin{equation}\label{ineq_expect}
e^{\frac{{\phi_K^-}(1)-|{\gamma}|}{q}} \E \Big[ \Big( A_{nq,q}/q\Big)^{-1/\beta}\Big]
 \leq \E \Big[ \Big(\int_0^n e^{-\beta K_u}\ud u \Big)^{-1/\beta}\Big]   
\leq e^{\frac{{\phi_K^+}(1)+|{\gamma}|}{q}}\E \Big[ \Big( A_{nq-1,q}/q\Big)^{-1/\beta}\Big].
\end{equation}
A direct application of Lemma \ref{lemtec} with 
$$ a(q)=d(q), \quad c^-(q)=e^{({\phi_K^-}(1)-|{\gamma}|)/q}, \quad \text{and} \quad c^+(q)=e^{({\phi_K^+}(1)+|{\gamma}|)/q}, $$
yields that $d(q)$ converges as $q\rightarrow \infty$  to a finite non-negative constant $d$.  Let us now prove that $d$ is positive. Let $(q_1,q_2)$ be in $\N^2$. According to \eqref{defd} and \eqref{ineq_expect} there exists $n \in \N$ such that 
$$0<e^{\frac{{\phi_K^-}(1)-|{\gamma}|}{q_1}}d(q_1)/2 \leq n^{3/2}e^{-n\phi_K(\tau)} \E \Big[ \Big(\int_0^n e^{-\beta K_u}\ud u \Big)^{-1/\beta}\Big]   
\leq 2 e^{\frac{{\phi_K^+}(1)+|{\gamma}|-\phi_K(\tau)}{q_2}}d(q_2) .$$
Letting $q_2$ go to infinity, we conclude that $\liminf_{q\to \infty}d(q)>0$.
Similar arguments  imply  the convergence of $b(q)$ to a positive constant.

Now, we prove the convergence of $c_2(q)$, $c_3(q)$ and $c_4(q)$. Again the proofs of the three cases are very similar, so  we only prove the second one. From Lemmas \ref{discret} and \ref{weak}, we know that for every $(n,q) \in {\N^2}$,
$$
\E \left[ F \Big( e^{\beta(|{\gamma}|/q +\sigma^{(-)}_{1/q})}  A_{nq,q}/q\Big)\right]
 \leq a_F(n) 
  \leq \E \left[ F \Big( e^{-\beta(|{\gamma}|/q +\sigma^{(+)}_{1/q})}  A_{nq-1,q}/q\Big)\right].
$$
Using (\ref{controle}), we obtain
\[
\begin{split}
 & F_{nq,q}+M \E\left[e^{-|{\gamma}|/q-\sigma^{(-)}_{1/q}} -1 \right] \E\left[  \Big(\frac{A_{nq,q}}{q}\Big)^{-\frac{1}{\beta}} \right] \\
 &\hspace{5cm} \leq  {a_F(n)}\leq \\
& \hspace{6cm} F_{nq-1,q}+M \E\left[ e^{|{\gamma}|/q+\sigma^{(+)}_{1/q}} -1 \right] \E\left[  \Big(\frac{A_{nq-1,q}}{q}\Big)^{-\frac{1}{\beta}} \right].
\end{split}
\]
Thus, dividing by $n^{-3/2 }\exp(n\phi_K(\tau))$ in the above inequality,  we  get the convergence using
Lemmas \ref{weak}, \ref{lemtec} and Equation (\ref{controle}) with  
$$ a(q)=c_3(q), \quad  c^-(q)=1-\frac{Md(q)(e^{({\phi_K^-}(1)-|{\gamma}|)/q}-1)}{c_3(q)} \quad c^+(q)=1+\frac{Md(q)(e^{({\phi_K^+}(1)+|{\gamma}|)/q}-1)}{c_3(q)} .$$
We then prove that $\lim_{q \to \infty}c_3(q)$ is positive using similar arguments as previously.
\end{proof}

\subsection{Proof of Theorem \ref{funcLev}}
\label{proofthp}

\begin{proof}[Proof of Theorem \ref{funcLev} a/ (i)]
Recall from Lemma II.2 in \cite{MR1771663} that the process $(K_t-K_{(t-s)^-}, 0\le s\le t)$ has the same law as $(K_s, 0\le s\le t)$. Then
\[
\int_0^t e^{-\beta K_s}ds=\int_0^t e^{-\beta K_{(t-s)}}ds=e^{-\beta K_t}\int_0^{t}e^{\beta K_t-\beta K_{(t-s)}}ds\stackrel{(d)}{=}e^{-\beta K_t}\int_0^{t}e^{\beta K_s}ds.
\]
We first note  that for every $q \in \N^*$ and $t \geq 2/q$, Lemma \ref{discret} leads to
\begin{eqnarray*} \E \left[ \left( \int_0^t e^{-\beta K_s}ds\right)^{-1/\beta} \right] & \leq & \E \left[ \left( \int_0^{2/q} e^{-\beta K_s}ds\right)^{-1/\beta} \right]\\
& \leq & { q^{1/\beta} e^{|\gamma|/q}\E\Big( e^{\sigma_{1/q}^{(+)}}(A_{1,q}^{(1)})^{-1/\beta} \Big)}\\
& \leq & q^{1/\beta} \exp \Big( \frac{\phi_K(1)+|\gamma|+\phi_K^+(1)}{q} \Big)<\infty,
  \end{eqnarray*}
where $\phi_K^+$ was defined  in (\ref{defphi}).  Hence using (\ref{escheer}), with $\lambda=1$, we have
\begin{eqnarray*}
 \E \left[ \left( \int_0^t e^{-\beta K_s}ds\right)^{-1/\beta} \right] =   \E \left[ e^{ K_t}\left( \int_0^t e^{\beta K_s}ds\right)^{-1/\beta} \right] =   e^{t\phi_K(1)} \E^{(1)} \left[ \left( \int_0^t e^{\beta K_s}ds\right)^{-1/\beta} \right]. 
\end{eqnarray*}
The above identity implies that the decreasing function $t \mapsto \E^{(1)} [ ( \int_0^t e^{\beta K_s}ds)^{-1/\beta} ] $ is finite for all $t>0$.  So it converges to a non-negative and finite limit $c_1$, as $t$ increases. This limit is positive, since under the probability $\P^{(1)}$, $K$ is still a L\'evy process with negative mean $\E^{(1)}(K_1)=\phi_K'(1)$ and according to  Theorem 1 in \cite{MR2178044}, we have
 $$ \int_0^\infty e^{\beta K_s}ds<\infty, \qquad \P^{(1)}\textrm{-a.s.}$$
Hence, we only need to prove 
\begin{equation}\label{equivstrong} a_F(t) \quad {\sim}\quad {C_F}  \E\Big[  \Big(  \int_0^t e^{-\beta K_s}ds\Big)^{-1/\beta}\Big], \quad \textrm{as }\quad t\to\infty.\end{equation}
Recall that  $\theta_{max}>1$ and $\phi_K'(1)<0$. So we can choose $ \varepsilon>0 $ such that (\ref{majepsilon}) holds, $1+\varepsilon<\theta_{max}$, $\phi_K(1+\varepsilon)<\phi_K(1)$ and $\phi_K^{\prime}(1+\varepsilon)<0$. Therefore
$$ 
\left| F\left(  \int_0^t e^{-\beta K_s}ds\right)-{C_F}\left(  \int_0^t e^{-\beta K_s}ds\right)^{-1/\beta} \right|\leq M \left(  \int_0^t e^{-\beta K_s}ds\right)^{-(1+\varepsilon)/\beta}. $$
In other {words}, it is enough  to show  
$$ \E \left[ \left( \int_0^t e^{-\beta K_s}ds\right)^{-(1+\varepsilon)/\beta} \right]=o(e^{t\phi_K(1)}), \quad \textrm{as }\quad t\to\infty. $$ 
From the Escheer transform (\ref{escheer}), with $\lambda=1+\varepsilon$, we deduce
\begin{eqnarray*}
 \E \left[ \left( \int_0^t e^{-\beta K_s}ds\right)^{-(1+\varepsilon)/\beta} \right] & = &  \E \left[ e^{ (1+\varepsilon)K_t}\left( \int_0^t e^{\beta K_s}ds\right)^{-(1+\varepsilon)/\beta} \right] \\
& = &  e^{t\phi_K(1+\varepsilon)} \E^{(1+\varepsilon)} \left[ \left( \int_0^t e^{\beta K_s}ds\right)^{-(1+\varepsilon)/\beta} \right] .
\end{eqnarray*}
Again from Lemma \ref{discret}, we obtain for $t\geq q/2$,
$$\E \left[ \left( \int_0^t e^{-\beta K_s}ds\right)^{-\frac{1+\varepsilon}{\beta}} \right]  \leq  q^{(1+\varepsilon)/\beta} \exp \Big( \frac{\phi_K(1+\varepsilon)+|\gamma|(1+\varepsilon)+\phi_K^+(1+\varepsilon)}{q} \Big)<\infty,
$$
{implying that the decreasing function $t \mapsto \E^{(1+\varepsilon)} [ ( \int_0^t \exp(\beta K_s)ds)^{-(1+\varepsilon)/\beta} ]$ is finite for all $t>0$.} This completes the proof.
\end{proof}

\begin{Rque}\label{th11GKV}
In the particular case when $\beta=1$, it is enough to apply Theorem 1.1 in 
\cite{MR1983172} to the geometric BPRE $(X_n, n\geq 0)$ defined in Remark \ref{th12GKV} to get the result. 
\end{Rque}

\begin{proof}[Proof of Theorem \ref{funcLev} a/ (ii), (iii), and b/]
The proofs are very similar for the {three}  regimes, for this reason we only focus on the proof of  the regime in a/(iii).

Let $ \varepsilon >0 $. Thanks to Lemma \ref{conv_cons}, we can choose $q \in \N^*$ such that $q\geq 1/\varepsilon$ and $(1-\varepsilon)c_3\leq c_3(q) \leq (1+\varepsilon)c_3$.
Then for every $t \geq 1$, the monotonicty of $F$ yields
$$ \E \Big[ F ( C_{\lfloor qt\rfloor,q} e^{\beta |\gamma|/q} /q)\Big]  
 \leq a_F  (t)   \leq  \E \Big[ F ( D_{\lfloor qt\rfloor-1,q} e^{-\beta |\gamma|/q} /q)\Big]. $$
Applying (\ref{controle}), we obtain :
\[ \begin{split}\Big| \E \Big[ F ( C_{\lfloor qt\rfloor,q} e^{\beta |\gamma|/q}/q )\Big]-F_{\lfloor qt\rfloor,q} \Big| & \leq (1-e^{-\varepsilon (|\gamma|-{\phi_K^-}(1))})M \E \Big[( A_{\lfloor qt\rfloor,q}/q )^{-1/\beta}\Big], \\
 \Big| \E \Big[ F ( D_{\lfloor qt\rfloor-1,q} e^{-\beta |\gamma|/q}/q )\Big]-F_{\lfloor qt-1\rfloor,q} \Big| & \leq (e^{\varepsilon (|\gamma|+{\phi_K^+}(1))}-1)M \E \Big[( A_{\lfloor qt\rfloor-1,q}/q )^{-1/\beta}\Big].  \end{split} \]
Taking  $t $ to infinity, it is clear  from Lemma \ref{weak} that both terms are bounded by 
\begin{equation} \label{def_h} l(\varepsilon)t^{-3/2}e^{t\phi_K(\tau)}=\Big[2M {d} (e^{\varepsilon( |\gamma|+{\phi_K^+}(1))}-e^{-\varepsilon (|\gamma|-{\phi_K^-}(1))})e^{-\varepsilon \phi_K(\tau)}\Big] t^{-3/2}e^{t\phi_K(\tau)} \end{equation}
where ${\phi_K^-}$ and ${\phi_K^+}$ are defined in (\ref{defphi}), and {${l}(\varepsilon)$ goes to $0$ when $\varepsilon$ decreases}. On the other hand, for $t$ large enough 
$$ (1-2\varepsilon)c_3t^{-3/2}e^{t\phi_K(\tau)} \leq F_{\lfloor qt\rfloor,q}  \leq a_F(t) \leq  F_{\lfloor qt\rfloor-1,q} \leq (1+2 \varepsilon)c_3t^{-3/2}e^{t\phi_K(\tau)},$$
which completes the proof of Theorem \ref{funcLev}.
\end{proof}

\section{Application to a cell division model}
\label{appli}
When the reproduction law has a finite second moment, the scaling limit of the {GW} process is a Feller diffusion with growth $g$ and diffusion part $\sigma^2$. That is to say, the stable case with $\beta=1$ and additional  drift term $g$. Such a process is also the scaling limit of birth and death process{es}. It gives a natural model for  populations which die and multiply  fast, randomly, without interaction. Such a model is considered in \cite{MR2754402} for parasites  growing in dividing cells. The cell divides  at constant rate $r$ and a random fraction $\Theta \in(0,1)$ of parasites   enters  the first daughter cell, whereas the remainder  enters  the second daughter cell. Following the infection in a cell line, the parasites grow as a Feller diffusion process and undergo a catastrophe when the cell divides.
{We denote by $N_t$ and $N_t^*$ the numbers of cells and infected cells at time $t$, respectively. We say that the cell population recovers when the asymptotic proportion of contaminated cells vanishes. If there is one infected cell at time $0$, $\E[N_t]=e^{rt}$ and $\E[N_t^*]=e^{rt}\P(Y_t>0),$} where\begin{align}
Y_t = &
1 +\int_0^t gY_s ds
+\int_0^t \sqrt{2\sigma^2 Y_s}dB_s + \int_0^t \int_{0}^{1} (\theta-1)
  Y_{s_-}  \rho(ds,d\theta)\label{defY}.
\end{align}
Here $B$ is a Brownian motion and $\rho(ds,d\theta)$ a Poisson random measure with intensity $2rds\P(\Theta\in d\theta)$.
Note that the intensity of $\rho$ is twice the cell division rate. This bias follows from the fact that if we pick an individual at random at time $t$, we are more likely to choose a lineage in which many division events have {occurred}. Hence the ancestral lineages from typical individuals at time t have a
division rate $2r$.

Corollary \ref{cor1} and Proposition  \ref{ppal_result}   with $\beta=1$, $\psi(\lambda)=-g\lambda +\sigma^2 \lambda$ and $\nu(dx)=2r\P(\Theta\in dx) $ imply the following result.
\begin{Cor}
\begin{enumerate}
 \item[a/]
 We assume that $ g < 2r\E\left[ \log(1/\Theta) \right]$. Then there exist positive constants $c_1,c_2,c_3$ such that
\begin{enumerate}
 \item[(i)] If $ g < 2r\E\left[ \Theta \log(1/\Theta) \right]$, then
$$ \E\left[N_t^{*}\right]\sim c_1 e^{gt},   \quad \textrm{as }\quad t\to\infty. $$
 \item[(ii)] If $ g = 2r\E\left[ \Theta \log(1/\Theta) \right]$,  then
$$ \E\left[N_t^{*}\right]\sim c_2 t^{-1/2}e^{gt},   \quad \textrm{as }\quad t\to\infty.$$
 \item[(iii)] If $ g > 2r\E\left[ \Theta \log(1/\Theta) \right]$,  then
$$ \E\left[N_t^{*}\right]\sim c_3  t^{-3/2}e^{\alpha t},   \quad \textrm{as }\quad t\to\infty.$$
where 
$\alpha = \min_{\lambda \in [0,1]}\{g\lambda +2r(\E[\Theta^\lambda]-1/2)\}<g.$
\end{enumerate}
\item[b/]We now assume  $ g = 2r\E\left[ \log(1/\Theta) \right]$, then there exists $c_4>0$ such that,
$$  \E\left[N_t^{*}\right]\sim c_4 t^{-1/2}e^{rt},   \quad \textrm{as }\quad t\to\infty.$$
\item[c/]Finally, if $ g > 2r\E\left[ \log(1/\Theta) \right]$, then there exists $0<c_5<1$ such that,
$$  \E\left[N_t^{*}\right]\sim c_5 e^{rt},   \quad \textrm{as }\quad t\to\infty.$$
\end{enumerate}
\end{Cor} 
Hence if $ g > 2r\E\left[ \log(1/\Theta) \right]$ (supercritical case c/), the mean number of infected cells is equivalent to the
 mean number of cells. In the critical case (b/), there are  somewhat fewer infected cells, owing to the additional square root term.  
In the strongly subcritical regime (a/ (i)), the mean number of infected cells is of the same order as the number of parasites.  This suggests that parasites do not accumulate in some infected cells.
 The asymptotic behavior in the two remaining cases is more complex.\\
\indent We stress the fact that fixing the growth rate $g$ of parasites and the cell division rate $r$, but making the law of the repartition $\Theta$ vary, it changes the asymptotic behavior of the number of infected cells.
For example, if we focus on random variables $\Theta$ satisfying $\P(\Theta=\theta)=\P(\Theta=1-\theta)=1/2$ for a given $\theta \in ]0,1/2{[}$, the different regimes can be described easily (see Figure \ref{cells}). 

If $g/r>\log 2$, the cell population either recovers or not, depending on the asymmetry of the parasite sharing. If $g/r\leq \log 2/2$, the cell population recovers but 
the speed of recovery increases with respect to the asymmetry of the parasite sharing, as soon as the weakly subcritical regime is reached. 
Such phenomena were known in the discrete time, discrete space framework (see \cite{MR2418235}), but the boundaries between the regimes are not the same, due to the bias in division rate
in the continuous setting. 
Moreover, we note that if $g/r \in (\log 2/2,\log 2)$, then parasites are in the weakly subcritical regime whatever the distribution of $\Theta$ on $]0,1[$. This phenomenon also only occurs in the continuous setting.

\begin{figure}[h]
\centering
\includegraphics[width=9cm,height=6cm]{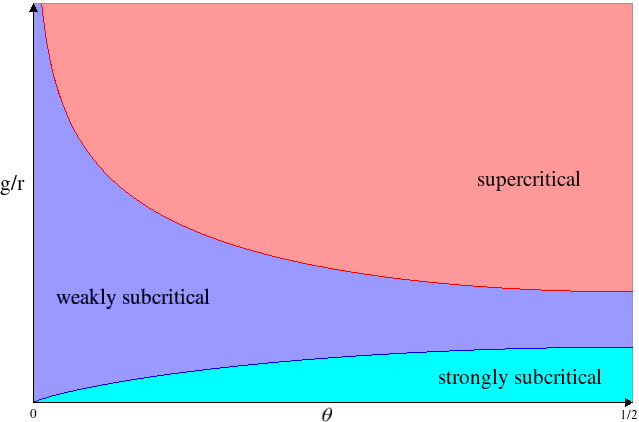}
\caption{Extinction regimes in the case $\P(\Theta=\theta)=\P(\Theta=1-\theta)=1/2$. Boundaries between the different regimes are given by $g/r =-\log (\theta(1-\theta)) $ (supercritical and subcritical) and $g/r =-\theta\log \theta-(1-\theta)\log (1-\theta) $ (strongly and weakly subcritical){.}}
\label{cells}
\end{figure}

\section{Auxiliary results}\label{annexe}
This section is devoted to  the technical results which are necessary for the previous proofs. 

\subsection{Existence and uniqueness of the backward ordinary differential equation}
The Laplace exponent of $\w{Z}$ in Theorem \ref{th1} is the solution of a backward ODE. The existence and uniqueness of this latter are stated and proved below.

\begin{Prop}
\label{existun} Let $\delta$ be in $\mathcal{BV}(\R^+)$.
Then the backward ordinary differential equation (\ref{eqnv}) admits a unique solution.
\end{Prop}

The proof relies on a classical approximation of the solution of (\ref{eqnv}) and the Cauchy-Lipschitz Theorem. {When there is no accumulation of jumps, the} latter provides the existence and uniqueness  of the solution between two successive jump times of $\delta$. The problem remains on the times where accumulation of jumps occurs.  Let us define the family of functions $\delta^n$ by deleting the small jumps of $\delta$,  
\[
\delta^n_t=\delta_t-\sum_{s\le t}\Big(\delta_s-\delta_{s-}\Big)\mathbf{1}_{\{|\delta_s-\delta_{s-}|<1/n\}}.
\]

We note that 
$\psi_0$ is continuous, and $ s \mapsto e^{gs+ \delta^n_s}$ is piecewise $ C^1$ on $\R^+$ with a finite number of discontinuities.   From the Cauchy-Lipschitz Theorem,  for every $n \in \N^*$  we can define a solution $v_t^n(.,\lambda,{\delta})$ continuous with c\`adl\`ag first derivative of the backward differential equation:
$$ \frac{\partial}{\partial s} v^n_t(s,\lambda, {\delta})=e^{gs+ \delta^n_s}\psi_0\big(e^{-gs- \delta^n_s}v_t^n(s,\lambda,{\delta})\big), \quad 0 \leq s \leq t,  \qquad v_t^n(t,\lambda,{\delta})=\lambda.$$

 We want to show that the sequence $(v_t^n(.,\lambda,\delta))_{n\geq 1}$ converges to a function  $v_t(.,\lambda,{\delta})$ which is  solution of (\ref{eqnv}). This follows from the next result. We   fix $t>0$ and define 
\begin{equation}\label{constants1} S:=\sup_{s \in [0,t],n\in \N^*} \Big\{ e^{gs+ \delta^n_s},e^{-gs- \delta^n_s} \Big\}. \end{equation}

\begin{Lem} \label{lemv} For every $\lambda > 0$, there exists a positive finite constant $C$ such that for all $0\leq \eta \le \kappa \leq  \lambda S,$
\begin{equation}\label{C}
0 \le \psi_0(\kappa)-\psi_0(\eta) \leq  C(\kappa-\eta).
\end{equation}
\end{Lem}

\begin{proof} 

First, we observe that $S$ is finite and that  for all $0\leq \eta < \kappa \leq \lambda S$, 
we have $0 \leq  e^{-\kappa x}-e^{-\eta x}+(\kappa-\eta)x \leq  (\kappa-\eta) x$ for 
$x\geq 0$ since $x \mapsto e^{-x}+x$ is increasing and $e^{-\kappa x}\leq e^{-\eta x}$. Moreover 
\begin{equation} \label{inegexp} 0 \leq e^{-x}-1+x\leq  x\wedge x^2,\end{equation}
and combining these  inequalities yields
\[
\begin{split}
& \psi_0(\kappa)-\psi_0(\eta) \\
&\quad =\sigma^2 (\kappa^2-\eta^2)+\int_1^\infty \Big(e^{-\kappa x}-e^{-\eta x}+(\kappa-\eta)x \Big)\mu(dx) \\
 &\quad \qquad +(\kappa -\eta)\int_0^1 x(1-e^{-\eta x})\mu(dx)+\int_0^1 \Big(e^{-(\kappa-\eta) x}-1+(\kappa-\eta)x \Big)e^{-\eta x}\mu(dx)\\
 &\quad \leq \sigma^2 (\kappa^2-\eta^2) +(\kappa-\eta)\int_1^\infty x\mu(dx)+(\kappa-\eta)\eta\int_0^1 x^2\mu(dx)+(\kappa-\eta)^2\int_0^1 x^2\mu(dx)\\
&\quad \leq   \Big[2\lambda S \sigma^2  +\int_1^\infty x\mu(dx)+\lambda S\int_0^1 x^2\mu(dx) \Big]  (\kappa-\eta),
\end{split}
\]
which proves Lemma \ref{lemv}.
\end{proof}

Next, we  prove the  existence and uniqueness result.

\begin{proof}[Proof of Proposition \ref{existun}]
We now prove that   $(v_t^n(s,\lambda,\delta),s \in [0,t])_{n \geq 0}$ is a Cauchy sequence. For  simplicity, we denote $v^n(s)=v_t^n(s,\lambda,\delta)$, and for all $v\geq 0$:
$$ \psi^n(s,v)= e^{gs+ \delta^n_s}\psi_0\big(e^{-gs- \delta^n_s}v\big) \quad \text{and} \quad \psi^\infty(s,v)= e^{gs+ \delta_s}\psi_0\big(e^{-gs- \delta_s}v\big) . $$

We have for any $ 0\leq s\leq t$ and $m,n\geq 1$:
\begin{eqnarray}\label{ctrl} |v^n(s)-v^m(s)|&=&\Big| \int_s^t \psi^n(u,v^n(u))du - \int_s^t \psi^m(u,v^m(u))du \Big|\\ \nonumber
&\leq & \int_s^t (R^n(u)+R^m(u))du+\int_s^t\Big|  \psi^\infty(u,v^n(u)) - \psi^\infty(u,v^m(u))\Big| du,\end{eqnarray}
where for  any $u \in [0,t]$,
\begin{eqnarray*}
R^n(u)&:=& \Big| \psi^n(u,v^n(u)) - \psi^\infty(u,v^n(u)) \Big|\\
&\leq & e^{gu+\delta^n_u}\Big| \psi_0\big(e^{-gu- \delta^n_u}v^n(u)\big)-\psi_0\big(e^{-gu- \delta_u}v^n(u)\big) \Big|+e^{gu}\psi_0\big(e^{-gu- \delta_u}v^n(u)\big) \Big| 
e^{ \delta^n_u}-e^{ \delta_u} \Big|.
\end{eqnarray*}
Moreover, from (\ref{constants1}) {to} (\ref{C}), we obtain
\begin{eqnarray*}
  R^n(u)&\leq & SC\lambda \Big| e^{- \delta^n_u}-e^{- \delta_u} \Big|+ e^{\vert g\vert t}\psi_0(\lambda S)\Big| e^{ \delta^n_u}-e^{ \delta_u} \Big| \\
& \leq & \Big( SC\lambda +e^{\vert g \vert t}\psi_0(\lambda S) \Big) \sup_{u \in [0,t]} \Big\{\Big| e^{- \delta^n_u}-e^{- \delta_u} \Big|, \Big| e^{ \delta^n_u}-
e^{ \delta_u} \Big| \Big\}:=s_n.
\end{eqnarray*}
Using similar arguments as above, we get from (\ref{C}),
$$ \Big|  \psi^\infty(u,v^n(u)) - \psi^\infty(u,v^m(u))\Big|\leq  CS^2 \Big| v^n(u)-v^m(u) \Big|. $$
From $(\ref{ctrl})$, we use   Gronwall's Lemma (see e.g. Lemma 3.2 in \cite{MR1112411}) with
$$ R_{m,n}(s)= \int_s^t R^n(u)du+\int_s^t R^m(u)du,$$
to  deduce that for all $0 \leq s \leq t$,
$$ |v^n(s)-v^m(s)|\leq R_{m,n}(s)+CS^2e^{CS^2(t-s)}\int_s^t R_{m,n}(u)du. $$
Recalling that $R_n(u)\leq s_n$ and $\int_s^t R^n(u)du \leq ts_n$ for $u\leq t$, we get for every $n_0 \in \N^*$,
$$ \underset{m,n\geq n_0, s \in [0,t]}{\sup}|v^n(s)-v^m(s)|\leq t \Big[1+CS^2e^{CS^2t}t\Big] \underset{m,n\geq n_0}{\sup}( s_n+s_m ).$$
Adding that $s_n\rightarrow 0$ ensures that  $(v^n(s),s \in [0,t])_{n\geq 0}$ is a Cauchy sequence under  the uniform norm. Then there exists  a continuous function $v$  on $[0,t]$ such that 
$v^n \to v, $ as $n$ goes to $\infty$.

Next, we prove that $v$ is solution of the Equation (\ref{eqnv}). As $\delta $ satisfies (\ref{constants1}), we have for any $s \in [0,t]$ and $n \in \N^*$:
\[
\begin{split}
 &\Big|v(s)-\int_s^t\psi^\infty (s,v(s))ds-\lambda \Big| \\
& \,\,\,\,\leq  \Big|v(s)-v^n(s)\Big|+\int_s^t\Big|\psi^\infty (s,v(s))-\psi^n (s,v(s))\Big|ds
+\int_s^t\Big|\psi^n (s,v(s))-\psi^n (s,v^n(s))\Big|ds\\
&\,\,\,\, \leq ts_n+ (1+CS^2)\sup \Big\{  \Big|v(s)-v^n(s)\Big|, s\in [0,t] \Big\},
\end{split}
\]
so that letting $n\rightarrow \infty$ yields
$\Big|v(s)-\int_s^t\psi^\infty (s,v(s))ds-\lambda \Big|=0.$
It  proves that  $v$ is solution of (\ref{eqnv}). The uniqueness  follows from   Gronwall's lemma.
\end{proof}

\subsection{An upper bound for $\psi_0$}

The study of the Laplace exponent of $\w{Z}$ in Corollary \ref{cor1} requires a fine control of the branching mechanism $\psi_0$.

\begin{Lem}\label{exg}
Assume that  the process ${(gt+\Delta_t, t\ge 0)}$ {goes} to $+\infty$ a.s. There exists a non-negative increasing function ${k}$ on $\R^+$ such that  for every $\lambda \ge 0$
$$ \psi_0(\lambda)\leq \lambda {k}(\lambda) \qquad\textrm{ and } \qquad\int_0^\infty {k}\Big( e^{-(gt+\Delta_t)} \Big) dt <\infty.$$
\end{Lem}

\begin{proof}
{The} inequality (\ref{inegexp}) implies that  for every $\lambda \geq 0$,
 \Bea
 \psi_0(\lambda) &\leq &\sigma^2\lambda^2+\int_0^{\infty} \left(\lambda^2{z}^2\mathbf{1}_{\{\lambda {z}\leq 1\}} +\lambda {z}\mathbf{1}_{\{\lambda {z}>1\}}\right)\mu(d{z}) \\
 & \leq & \left(\sigma^2+ \int_0^1{z}^2\mu(d{z}) \right)\lambda^2+\lambda^2\mathbf{1}_{\{\lambda <1\}}\int_1^{1/\lambda}{z}^2\mu(d{z})+\lambda \int_{1/\lambda}^{\infty} {z}\mu(d{z}). \\
 \Eea
Now, using  condition (\ref{xlogx}) we obtain the existence of a positive constant $c$ such that
\begin{eqnarray*}
 \lambda \int_{1/\lambda}^{\infty} {z}\mu(d{z}) \leq  \lambda\log^{{-(1+\varepsilon)}}(1+1/\lambda) \int_{1/\lambda}^{\infty} {z}\log^{{1+\varepsilon}}(1+{z})\mu(d{z}) \leq  c\lambda\log^{{-(1+\varepsilon)}}(1+1/\lambda).
\end{eqnarray*}
Next, let us introduce the function $f$, given by
$$ f({z})={z}^{-1}\log^{{1+\varepsilon}}(1+{z}), \quad \textrm{for } {z} \in [1,\infty). $$
If we derivate the function $f$, we deduce  that there exists a positive real number $A> 1$ such that $f$ is decreasing on $[A,\infty)$. Therefore, for every $\lambda<1/A$, 
\begin{eqnarray*}
\int_{{A}}^{1/\lambda}\lambda^2{z}^2\mu(d{z})&{=}&\lambda \log^{{-(1+\varepsilon)}}\left(1+1/\lambda\right)f\left(1/\lambda\right) \int_A^{1/\lambda}\frac{{z}\log^{{1+\varepsilon}}(1+{z})}{f({z})} \mu(d{z})\\
&\le&\lambda\log^{{-(1+\varepsilon)}}\left(1+1/\lambda\right) \int_A^{1/\lambda}{z}\log^{{1+\varepsilon}}(1+{z}) \mu(d{z}).\\
%& &+ \lambda^2 A \int_1^\infty x \mu(dx).\\
\end{eqnarray*}
Adding that {$
\lambda^2 \int_1^{A}{z}^2\mu(d{z})\leq \lambda^2 A \int_1^\infty {z} \mu(d{z})$} and using again  condition (\ref{xlogx}), we deduce that there exists a positive constant $c'$ such that  for every $\lambda \geq 0$,
$$ \psi_0(\lambda)\leq c' \Big(  \lambda^2+\lambda \log^{{-(1+\varepsilon)}}(1+1/\lambda)  \Big).$$
Since  $\lambda^2$ is negligible with respect to $\lambda \log^{{-(1+\varepsilon)}}(1+1/\lambda)$ when $\lambda$ is close enough to  $0$ or infinity, we 
conclude that there exists a positive constant $c''$ such that
$$ \psi_0(\lambda)\leq c''\lambda \log^{{-(1+\varepsilon)}}(1+1/\lambda).  $$
Defining the function ${k}({z})= c'' \log^{{-(1+\varepsilon)}}(1+1/{z})$, for ${z}> 0$, we get that:
$$ {k}\Big( e^{-(gt+\Delta_t)} \Big)\sim c'' \log^{-(1+\varepsilon)}(2), \quad (t \to 0), $$
thus the integral of $k( \exp(-gt-\Delta_t))$ is finite in a neighborhood  of zero, and
$$ 0 \leq  \int_1^\infty {k}\Big( e^{-(gt+\Delta_t)} \Big) dt \leq c''  \int_1^\infty e^{-(gt+\Delta_t)}(gt+\Delta_t)^{-(1+\varepsilon)}  dt,$$
which is finite since the process $(gt+\Delta_t, t\ge 0)$ drifts $+\infty$ and has finite first moment. This completes the proof.
\end{proof}

\subsection{Extinction versus explosion}

We now verify that the process $(Y_t)_{t\ge 0}$ can be properly renormalized as $t\rightarrow \infty$ on the non-extinction event. We use a classical branching argument.

\begin{Lem} \label{w0}
Let  $Y$ be a non-negative Markov process satisfying the branching property.  
We also assume  that there exists a positive  function $a_t$ such that for every $x_0>0$, there exists a non-negative finite random variable $W$ such that
$$a_tY_t\xrightarrow[t\to\infty]{} W \quad \text{a.s},\qquad  \P_{x_0}(W>0)>0, \qquad a_t \stackrel{t\rightarrow\infty}{\longrightarrow} 0. $$
Then 
$$\{W=0\}=\Big\{ Y_t\xrightarrow[t\to\infty]{} 0 \Big\} \qquad \P_{x_0} \quad a.s. $$
\end{Lem}

\begin{proof}
First, we prove that
\be
\label{limsup}
\P_{x_0}(\limsup_{t\to\infty} Y_t =\infty \ \vert \ \limsup_{t\to\infty} Y_t >0)=1.
\ee
Let $0<x\leq  x_0 \leq A$ be fixed. Since $a_t \rightarrow 0$ and $\P_{x}(W>0)>0$,  there exists $t_0>0$ such that
$\alpha:=\P_{x}(Y_{t_0}\geq A)>0.$
By the branching property, the process is stochastically monotone as a function of its initial value. Thus, for every $y \geq x$ (including $y=x_0)$,
 $$\P_{y}(Y_{t_0}\geq A)\geq \alpha>0.$$
We define recursively  the  stopping times
$$T_0:=0, \qquad T_{i+1}=\inf \{t \geq T_i +t_0 : Y_t \geq x\} \qquad (i\geq 0){.}$$
For any  $i \in \N^*$, the strong Markov property implies
$$ \P_{x_0}(Y_{T_i+t_0}\geq A \ | \ (Y_{t} : t\leq T_i),  \  T_i<\infty)\geq \alpha.  $$
Conditionally on $\{\limsup_{t\to\infty} Y_t > x\}$,  the stopping times $T_i$ are finite a.s. and  for all $0<x\leq x_0\leq A$,
$$\P_{x_0}( \forall   i \geq 0 : Y_{T_i+t_0} < A, \ \limsup_{t\to\infty} Y_t > x)=0.$$
Then, $\P_{x_0}( \limsup_{t\rightarrow \infty} Y_t < \infty, \ \limsup_{t\to\infty} Y_t > x)=0.$
Now since $\{ \limsup_{t\to\infty} Y_t >0\}= \cup_{x\in (0,x_0]} \{\limsup_{t\to\infty} Y_t > x\}$, we get (\ref{limsup}). 

Next, we consider the stopping times $T_n=\inf\{ t\geq 0 : Y_t\geq n\}$. The strong Markov property and branching property imply
\begin{eqnarray*}\P_{x_0}(W=0;  T_n<\infty) = \E_{x_0}\Big( \mathbf{1}_{T_n<\infty} \P_{Y_{T_n}}(W=0) \Big) 
\leq  \P_n(a_t Y_t\underset{t\rightarrow\infty}{\longrightarrow} 0)= \P_1(a_t Y_t\underset{t\rightarrow\infty}{\longrightarrow} 0)^n,\end{eqnarray*}
which goes to zero as $n \rightarrow \infty$, since $\P_1(a_t Y_t\stackrel{t\rightarrow\infty}{\longrightarrow} 0)=\P_1(W=0)<1$. Then,  
$$0=\P_{x_0}(W=0; \forall n :  T_n<\infty)=\P_{x_0}(W=0, \limsup_{t\rightarrow\infty} Y_t=\infty)=\P_{x_0}(W=0, \limsup_{t\rightarrow\infty} Y_t>0),$$
where the last identity comes from  $(\ref{limsup})$. This completes the proof.
\end{proof}

\subsection{A Central limit theorem}\label{subTCL}

We need the following central limit theorem for L\'evy processes in Corollary \ref{TCL}. 

\begin{Lem} Under the assumption (\ref{CLT})
we have
\[
\frac{gt+\Delta_t-\textbf{m}t}{\rho \sqrt{t}}\xrightarrow[t\to \infty]{d}{ \mathcal{N}}(0,1).
\]
\end{Lem}
\begin{proof}
For simplicity, let $\eta$ be the image measure of $\nu$ under the mapping $x\mapsto e^x$. Hence,  assumption (\ref{CLT}) is equivalent to $\int_{|x|\ge 1} x^2 \eta(\ud x) <\infty$, or $\E[\Delta_1^2]<\infty$. 

We define
$T(x)=\eta\big((-\infty, -x)\big)+\eta\big((x,\infty)\big)$ and $
U(x)=2\int_0^x yT(y) \ud y,$
and  assume that $T(x)>0$ for all $x>0$.
According to Theorem 3.5 in Doney and Maller \cite{MR1922446} there exist two functions $a(t), b(t)> 0$ such that
\[
\frac{gt+\Delta_t-a(t)}{b(t)}\xrightarrow[t\to \infty]{d}{\mathcal{N}}(0,1),
\quad\textrm{ 
if and only if}\quad 
\frac{U(x)}{x^2T(x)}\xrightarrow[x\to \infty]{}\infty.
\]
If the above condition is satisfied, then $b$ is regularly varying with index $1/2$ and it may be chosen to be strictly increasing to $\infty$ as $t\to \infty$. Moreover $b^2(t)=tU(b(t))$ and 
$a(t)=tA(b(t))$, where 
\[
A(x)=g+\int_{\{|z|<1\}}z\eta(\ud z)+\eta\big((1,\infty)\big)-\eta\big((-\infty, -1)\big)+\int_1^x\Big(\eta\big((y,\infty)\big)-\eta\big((-\infty, -y)\big)\Big)\ud y.
\]
Note that under our assumption $x^2T(x)\to 0$, as $x\to \infty$. Moreover, note
\[
U(x)=x^2T(x)+\int_{(-x,0)} z^2\eta(\ud x)+\int_{(0,x)} z^2\eta(\ud x),
\]
and 
\[
A(x)=g+\int_{\{|z|<x\}}z\eta(\ud z)+x\Big(\eta\big((x,\infty)\big)-\eta\big((-\infty, -x)\big)\Big).
\]
Hence assumption (\ref{CLT}) implies that 
\[
U(x)\xrightarrow[x\to\infty]{}\int_{(-\infty,\infty)} z^2\eta(\ud z)=\rho^2,\quad A(x)\xrightarrow[x\to\infty]{}g+\int_{\mathbb{R}} z\eta(\ud z)=\textbf{m},
\]
Therefore, we deduce
$U(x)/(x^2T(x))\to\infty$ as $x\to\infty$,
  $b(t)\sim\rho \sqrt{t}$ and $a(t)\sim\textbf{m}t$, as $t\to \infty$. 
  
  Now assume that $T(x)=0$, for  $x$ large enough. Define
  \[
  \Psi(\lambda, t)=-\log \mathbb{E}\left[\exp\left\{i\lambda\left(\frac{gt+\Delta_t-a(t)}{b(t)}\right)\right\}\right],
  \]
  where the functions $a(t)$ and $b(t)$ are defined as above.  Hence since the process $(\Delta_t,t\ge0)$ is of bounded variation, from  the definition of $a(t)$ and the L\'evy-Khintchine formula  we deduce
  \[
  \begin{split}
 \Psi(\lambda, t)&=-i\lambda\left( \frac{g t}{b(t)}-\frac{a(t)}{b(t)}\right)+t\int_{\mathbb{R}}\Big(1-e^{\frac{i\lambda }{b(t)}x}\Big)\eta(\ud x)\\
 &=t\int_{\{|x|<b(t)\}}\Big(1-e^{\frac{i\lambda }{b(t)}x}+\frac{i\lambda}{b(t)}x+\frac{(i\lambda)^2}{2b^2(t)}x^2\Big)\eta(\ud x) -\frac{t(i\lambda)^2}{2b^2(t)}\int_{\{|x|<b(t)\}}x^2\eta(\ud x)\\
  &\hspace{2cm}+t\int_{\{|x|\geq b(t)\}}\Big(1-e^{\frac{i\lambda }{b(t)}x}\Big)\eta(\ud x)+i\lambda t\Big(\eta(b(t), \infty)-\eta(- \infty, -b(t))\Big).
  \end{split}
  \]
  Since $T(x)=0$ for all $x$ large, $b(t)\to\infty$ and $t^{-1}b^2(t)\to{\rho^2}$, as $t\to\infty$, therefore
  \[
 \Psi(\lambda, t)\xrightarrow[t\to\infty]{}\frac{\lambda^2}{2}, 
  \]
  which implies the result thanks to Lévy's Theorem.
\end{proof}

\subsection{A technical Lemma} \label{techlem}

We now prove {a technical lemma} that  is  needed in the proofs of Section \ref{demothm}.

\begin{proof}[Proof of Lemma \ref{lemtecaussi(quel_nom_original...)}]
To obtain (\ref{majepsilon}), it is enough to choose $\varepsilon\leq 1 $ as we {assume} in (\ref{defF}) that $\varsigma \geq 1 $.

In order to  prove (\ref{controle}), we first define the function $\tilde{h}: x \in \R^+ \mapsto (1+x)^{1-\varsigma}h(x) $ and let $0\le x \le y$. Then,
\begin{eqnarray} \label{decoF}
\frac{F(x)-F(y)}{C_F} &\leq& \Big((x+1)^{-{1}/{\beta}}-(y+1)^{-{1}/{\beta}}\Big)+(1+y)^{-{1}/{\beta}-1}\Big|\tilde{h}(x)-\tilde{h}(y)\Big| \nonumber\\
&& \quad + \Big|\tilde{h}(x)\Big|\Big((1+x)^{-{1}/{\beta}-1}-(1+y)^{-{1}/{\beta}-1}\Big).
\end{eqnarray}
We deal with the second term of the right hand side. Denoting by $k$ the Lipschitz constant of $\tilde{h}$ and  applying  the Mean Value Theorem to $z \in \R_+ \mapsto (z+1)^{-1/\beta}$ on $[x,y]$, we get
$$(1+y)^{-{1}/{\beta}-1}\Big|\tilde{h}(x)-\tilde{h}(y)\Big| \leq k(y+1)^{-1/\beta -1} (y-x) \leq  k\beta\left((x+1)^{-1/\beta}-(y+1)^{-1/\beta}\right). $$
Moreover, as $\beta \in (0,1]$, we have the following inequalities :
$$
\left(\frac{1+y}{1+x}\right)^{1+1/\beta}-1 \leq \left(\left(\frac{1+y}{1+x}\right)^{1/\beta}-1 \right)\left(\frac{1+y}{1+x}-1 \right)\leq \left(\left(\frac{y}{x}\right)^{1/\beta}-1 \right)2\frac{1+y}{1+x}$$
Dividing by $(1+y)^{{1}/{\beta}+1}$ and using $(1+y)/[(1+x)(1+y)^{{1}/{\beta}+1}] \leq y^{-1/\beta}$  {yield}
$$
(1+x)^{-{1}/{\beta}-1}-(1+y)^{-{1}/{\beta}-1}\leq 
2 \Big( x^{-1/\beta}-y^{-1/\beta} \Big).$$
Similarly $(1+x)^{-{1}/{\beta}}-(1+y)^{-{1}/{\beta}} \leq x^{-1/\beta}-y^{-1/\beta}$ and  equation (\ref{decoF})  give us
\begin{eqnarray*}
 0 \leq F(x)-F(y) &\leq& C_F(1+2[\|h\|_\infty+k\beta]) \Big( x^{-1/\beta}-y^{-1/\beta} \Big).
\end{eqnarray*}
This completes the  proof.
\end{proof}

\subsection{Approximations of the survival probability for $\nu(0,\infty)=\infty$}
\label{details}

Finally, we prove  Corollary \ref{cor} in the case when $\nu(0,\infty)=\infty.$

\begin{proof}[End of the proof of Corollary \ref{cor}]

We  let $A^{\varepsilon_1,\varepsilon_2}=(0,1-\varepsilon_1)\cup(1+\varepsilon_2,\infty)$, where $0<1-\varepsilon_1<1<1+\varepsilon_2$ and define  the Poisson random  measure $N_1^{\varepsilon_1,\varepsilon_2}$ as the restriction of $N_1$ to   ${\R^+\times A^{\varepsilon_1,\varepsilon_2}}$. We denote by  $\ud t \nu^{\varepsilon_1,\varepsilon_2}(\ud {m})$ for its intensity
 measure, where 
$\nu^{\varepsilon_1,\varepsilon_2}(\ud {m})={\mathbf 1}_{\{{m}\in A^{\varepsilon_1,\varepsilon_2}\}}\nu(\ud {m})$, 
and the corresponding L\'evy process $\Delta^{\varepsilon_1,\varepsilon_2}$ is defined by
$$
 \Delta^{\varepsilon_1,\varepsilon_2}_t=\int_0^t \int_{(0,\infty)}\log {m} \hspace{.1cm} N_1^{\varepsilon_1,\varepsilon_2}(\ud s,\ud {m}).
$$
We also consider the CSBP's $Y^{\varepsilon_1,\varepsilon_2}$ (resp $Y^{\varepsilon_1,\varepsilon_2,-}$ and $Y^{\varepsilon_1,\varepsilon_2,+}$) with branching mechanism $\psi$ (resp. $\psi_-$ and $\psi_+$) and the same catastrophes $\Delta^{\varepsilon_1,\varepsilon_2}$ via (\ref{EDS}). Since $\nu^{\varepsilon_1,\varepsilon_2}(0,\infty)<\infty$, from the first step we have
$u^{\varepsilon_1,\varepsilon_2}_{+,t}(\lambda) \leq u^{\varepsilon_1,\varepsilon_2}(t,\lambda) \leq u^{\varepsilon_1,\varepsilon_2}_{-,t}(\lambda),$
where as expected  $\E[\exp\{-\lambda Y^{\varepsilon_1,\varepsilon_2,*}_t\}]=\exp\{-u^{\varepsilon_1,\varepsilon_2}_{*,t}(\lambda)\}$ for each $* \in \{ +,\emptyset,- \}$.

Similarly, let $A^{\varepsilon_1}=(0,1-\varepsilon_1)\cup(1,\infty)$ and define  the Poisson random  measure $N_1^{\varepsilon_1}$ as the restriction of $N_1$ to   ${\R^+\times A^{\varepsilon_1}}$ with intensity measure $\ud t \nu^{\varepsilon_1}(\ud {m})$, where 
$\nu^{\varepsilon_1}(\ud {m})={\mathbf 1}_{\{{m}\in A^{\varepsilon_1}\}}\nu(\ud {m})$. Let us fix $t$ in $\R^*_+$, 
and define $Y^{\varepsilon_1}$ as the unique strong solution of 
\begin{equation}\label{Yvarepsilon1}\begin{split}
Y^{\varepsilon_1}_t=Y_0+\int_0^t gY^{\varepsilon_1}_s \ud s +\int_0^t\sqrt{2\sigma^2 Y^{\varepsilon_1}_s} \ud B_s
&+\int_0^t\int_{[0,\infty)}\int_0^{Y^{\varepsilon_1}_{s-}} z\widetilde{N}_0(\ud s, \ud z, \ud u) \\
&+\int_0^t\int_{[0,\infty)} \Big({m}-1\Big) Y^{\varepsilon_1}_{s-}N_1^{\varepsilon_1}(\ud s,\ud {m}). 
\end{split} \end{equation}
We already know from Theorem \ref{th1} that Equation (\ref{Yvarepsilon1}) has a unique non-negative strong solution. Moreover, from Theorem 5.5 in 
\cite{RePEc:eee:spapps:v:120:y:2010:i:3:p:306-330} and the fact that $N_1^{\varepsilon_1}$ has the same jumps as $N_1^{\varepsilon_1,\varepsilon_2}$ plus additional jumps greater  than one, we conclude 
$$ Y^{\varepsilon_1,\varepsilon_2}_t \leq Y^{\varepsilon_1}_t,\qquad \textrm{ a.s.}$$
Using assumption (\ref{condh1}), we can  apply Gronwall's Lemma to the non-negative function $t \mapsto \E[Y^{\varepsilon_1}_t-Y^{\varepsilon_1,\varepsilon_2}_t]$ and  obtain 
$$ \E\Big[\big|Y^{\varepsilon_1,\varepsilon_2}_t-Y^{\varepsilon_1}_t\big|\Big] \xrightarrow[\varepsilon_2 \to 0]{} 0. $$

Now, since $Y^{\varepsilon_1,\varepsilon_2}$ is decreasing with $\varepsilon_2$, we finally get,  $ Y^{\varepsilon_1,\varepsilon_2}_t \xrightarrow[]{a.s.} Y^{\varepsilon_1}_t,$  as $\varepsilon_2\to 0$. 
Using similar arguments as above for $Y^{\varepsilon_1,\varepsilon_2,+}$ and $Y^{\varepsilon_1,\varepsilon_2,-}$, we deduce
$$u^{\varepsilon_1}_{+,t}(\lambda) \leq u^{\varepsilon_1}(t,\lambda) \leq u^{\varepsilon_1,}_{-,t}(\lambda).$$
In order to complete the proof, we let   $\varepsilon_1$ tend to $0$.
\end{proof}

{\bf Acknowledgements:} {\sl{The authors would like to thank   Jean-François Delmas, Sylvie M\'el\'eard {and Vladimir Vatutin} for their  reading of 
this paper and  their suggestions. They also want  to thank  Amaury Lambert for fruitful discussions at the beginning of this work, so as the two 
anonymous referees for several corrections and improvements. This work  was partially funded by project MANEGE `Mod\`eles
Al\'eatoires en \'Ecologie, G\'en\'etique et \'Evolution'
09-BLAN-0215 of ANR (French national research agency),  Chair Modelisation Mathematique et Biodiversite VEOLIA-Ecole Polytechnique-MNHN-F.X. and 
the professorial chair Jean Marjoulet.}
}

\end{document}